\documentclass[12pt]{amsart}
\usepackage{amsmath, amsthm, amssymb,cite,enumitem}
\usepackage{fullpage}
\usepackage{color}
\usepackage{hyperref}
\usepackage[makeroom]{cancel}
\usepackage{amsmath}
\usepackage{amssymb}
\usepackage{amsrefs}
\usepackage{enumitem}

\newcommand{\ol}[1]{\overline{#1}}

\newcommand{\tx}[1]{\textrm{#1}}

\newcommand{\wh}[1]{\widehat{#1}}
\newcommand{\wt}[1]{\widetilde{#1}}

\newcommand\R{\mathbb{R}}
\newcommand\Z{\mathbb{Z}}

\newcommand{\dd}{d}
\newcommand{\ddt}{\partial_t}
\newcommand{\del}{\delta}

\newcommand{\dt}{\,\dd t}
\newcommand{\dx}{\,\dd x}
\newcommand{\dxi}{\,\dd\xi}
\newcommand{\dy}{\,\dd y}
\newcommand{\dz}{\,\dd z}
\newcommand{\eps}{\varepsilon}

\newcommand{\bnorm}[1]{\big\lVert#1\big\rVert}
\newcommand{\hs}{\mathfrak{I}_2}
\newcommand{\ls}{\textup{LS}}
\newcommand{\norm}[1]{\left\lVert#1\right\rVert}
\newcommand{\op}{\mathfrak{I}_\infty}
\newcommand{\snorm}[1]{\lVert#1\rVert}
\newcommand{\tc}{\mathfrak{I}_1}

\newcommand{\tr}{\operatorname{tr}}
\newcommand{\Err}{Err}

\newtheorem{theorem}{Theorem}[section]
\newtheorem{lemma}[theorem]{Lemma}
\newtheorem{corollary}[theorem]{Corollary}
\newtheorem{proposition}[theorem]{Proposition}

\theoremstyle{definition}

\newtheorem{remark}[theorem]{Remark}

\numberwithin{equation}{section}

\allowdisplaybreaks

\begin{document}

\title{A-priori estimates for generalized \\ Korteweg--de Vries equations in $H^{-1}(\R)$}

\author[M.~Ifrim]{Mihaela Ifrim}%${}^*$} 
%\footnotemark{Corresponding author}
\address{Department of Mathematics, University of Wisconsin--Madison, Madison, WI 53706, USA}
\email{ifrim@math.wisc.edu}
%\thanks{${}^*$Corresponding author}

\author[T.~Laurens]{Thierry Laurens}
\address{Department of Mathematics, University of Wisconsin--Madison, Madison, WI 53706, USA}
\email{laurens@math.wisc.edu}

\begin{abstract}
We prove local-in-time a-priori estimates in $H^{-1}(\mathbb{R})$ for a family of generalized Korteweg--de Vries equations.  This is the first estimate for any non-integrable perturbation of the KdV equation that matches the regularity of the sharp well-posedness theory for KdV.  In particular, we show that our analysis applies to models for long waves in a shallow channel of water with an uneven bottom.

The proof of our main result is based upon a bootstrap argument for the renormalized perturbation determinant coupled with a local smoothing norm.
\end{abstract}

\maketitle
\tableofcontents

\section{Introduction}

The Korteweg--de Vries equation
\begin{equation}
\ddt u = - u''' + 6uu'
\tag{KdV}\label{KdV}
\end{equation}
(where $u' = \partial_xu$) was originally derived as a model for the propagation of waves in a shallow channel of water~\cite{Korteweg1895}.  However, over the next century, \eqref{KdV} has proved to be a fundamental model for nonlinear dispersive waves, with applications to a wide variety of physical systems spanning many fields of science (see, for example,~\cite{Crighton1995}).

Mathematically, \eqref{KdV} has also played a central role in our understanding of dispersive equations.  In particular, the \eqref{KdV} equation was the first discovered example of a completely integrable PDE, a feature that has been heavily exploited in some works.  This discovery sparked a pursuit into the well-posedness and dispersive behavior for this system over the next 50 years, and generated a long list of cutting-edge techniques ~\cites{Bona1976,Kato1975,Saut1976,Temam1969,Tsutsumi1971,Kenig1991,Bourgain1993,Christ2003,Kenig1996,Colliander2003,Guo2009,Kishimoto2009, KT, MR4706572}, some of which rely on complete integrability, and some of which do not. On the integrable side this effort culminated in global well-posedness for initial data in $H^{-1}$ on the line and the circle~\cites{Killip2019,Kappeler2006}, a result that is sharp in the class of $H^s$ spaces for both geometries~\cites{Molinet2011,Molinet2012}.

Although \eqref{KdV} lied at the center of this effort, it also served as a testing ground for the introduction of new tools and innovative techniques.  After their introduction, many of these methods (both integrable and non-integrable) were later adapted to broader classes of dispersive equations.  However, this important step has yet to be achieved for the completely integrable methods used to prove the sharp well-posedness results~\cites{Killip2019,Kappeler2006}.

One physically important class of examples are the KdV equations with variable bottom.  As \eqref{KdV} has proved to be an effective model for the propagation of waves in a shallow channel of water, many authors have introduced generalizations to describe an uneven bottom~\cites{Dingemans1997,Grimshaw1999,Groesen1993,Johnson1973,Miles1979,Pudjaprasetya1996,Pudjaprasetya1999,Yoon1994,Israwi2010,Lannes2013,Tian2001}.  For concreteness, consider the model
\begin{equation}
\partial_t u = - b^5 u''' + 6 uu' - 4 b u' - 6 b' u,
\label{KdVvb}
\end{equation}
which was derived and rigorously justified in~\cite{Israwi2010}.  Here, $b(x)$ is defined in terms of a function $c(x)$ which describes the bottom of the channel:
\begin{equation*}
b(x) := \sqrt{1 - c(x)}.
\end{equation*} 
When $c(x)$ is a small Schwartz function, the equation \eqref{gKdV} closely resembles \eqref{KdV}.  Nevertheless, our understanding of the well-posedness problem is comparatively lacking.

After a change of variables (see \eqref{cov} below), the equation \eqref{KdVvb} can be put in the form 
\begin{equation}
\ddt u = - u''' + 6uu' + (a_1u')' + a_2u^2 + a_3u' + a_4u 
\tag{gKdV}\label{gKdV}
\end{equation}
for certain coefficients $a_j$.  In this work, we will assume that the coefficients $a_j(t,x)$ are given functions that are sufficiently regular and localized in space, and study the corresponding solutions $u:[-T,T]\times\R\to\R$.

Our main contribution is the following local-in-time a-priori estimate in $H^{-1}$:
\begin{theorem}[A-priori estimate]\label{t:energy}
There exists $\epsilon >0$ so that, if the coefficients are given smooth functions that satisfy the decay bounds uniformly for $|t|\leq T$ and $x\in\R$:
\begin{align}
|a_j(t,x)| + |\partial_xa_j(t,x)| &\leq \epsilon (1+x^2)^{-1} \quad\text{for }j=1,2,3 , 
\label{hyp 1}\\
|a_4(t,x)| + |\partial_xa_4(t,x)| &\lesssim (1+x^2)^{-1},
\label{hyp 2}
\end{align}
then for any $A>0$ there exist constants $C,T > 0$ so that for any smooth solution $u$ to \eqref{gKdV} in $[-T,T]$ whose initial data satisfies 
\begin{equation}\label{data}
\norm{u(0)}_{H^{-1}} \leq A,
\end{equation}
satisfies the uniform bound
\begin{equation}
 \sup_{|t|\leq T} \norm{u(t)}_{H^{-1}} \leq C .
\label{ap est}
\end{equation}
\end{theorem}
\begin{remark}
The lifespan $T$ for the bounds in the Theorem~\ref{t:energy} is more accurately identified in the proof as 
\[
T\gtrsim (1+A)^{-4}.
\]
On the other hand one might expect that $C$ should be comparable to $A$; however this is not the case because in the proof of the Theorem~\ref{t:energy} we work with a weighted form of the $H^{-1}$ norm. So instead we get the weaker bound
\[
C  \lesssim A (1+A)^2.
\]

\end{remark}

\begin{remark}
The assumption that the coefficients $a_j$ are smooth in  the above theorem is purely qualitative, and only serves 
to ensure that we can talk about smooth solutions to~\eqref{gKdV}. Alternatively, if we want to only assume that 
the coefficients satisfy \eqref{hyp 1} and \eqref{hyp 2},
then the conclusion of the Theorem~\ref{t:energy} would remain valid for 
rougher solutions, e.g.\ in $L^2$,  once local well-posedness is known.

\end{remark}

We do not claim that the conditions \eqref{hyp 1}--\eqref{hyp 2} in  Theorem~\ref{t:energy} are sharp.  In fact, our proof will require slightly weaker hypotheses (see \eqref{a}--\eqref{d} below for details).  The main thrust of this work is to match the sharp regularity of the well-posedness theory for \eqref{KdV}, and so any hypotheses on the coefficients that work constitutes significant progress.  In particular, \eqref{ap est} rules out strong forms of instantaneous norm inflation in $H^{-1}$.

Applying our general result to the equation~\eqref{KdVvb}, we obtain:
\begin{corollary}\label{t:KdV vb}
There exists a constant $\epsilon > 0$ so that, if $c:\R\to\R$ is a smooth function that satisfies the small pointwise bounds
\begin{equation*}
|\partial_x^jc(x)| \leq \epsilon (1+x^2)^{-1} \quad\text{for }j=0,1,\dots,4 , 
\end{equation*}
then for any $A>0$ there exist constants $C,T > 0$ so that
for any smooth solution $u$ to \eqref{KdVvb} in $[-T,T]$ whose initial data satisfies
\begin{equation*}
\norm{u(0)}_{H^{-1}} \leq A ,
\end{equation*}
satisfies the uniform bound
\begin{equation*}
 \sup_{|t|\leq T} \norm{u(t)}_{H^{-1}} \leq C .
\end{equation*}
\end{corollary}

We contend that for each of the numerous applications of \eqref{KdV}, our a-priori estimate \eqref{ap est} could also be employed to describe small physical imperfections using similar methods.

Another main way in which \eqref{gKdV} arises naturally is in the study of localized perturbations of a background solution to \eqref{KdV}: if we take our solution $u = q+V$ to be a given background wave $V$ plus a perturbation $q$, then the equation for $q$ often takes the form \eqref{gKdV}.  For example, if $V(t,x)$ solves \eqref{KdV} then $q$ solves
\begin{equation}
\ddt q = - q''' + 6qq' + 6Vq' + 6V'q ,
\label{KdVwp}
\end{equation}
and if $V(x)$ is a fixed background profile then a forcing term $-V'''+6VV'$ is added to the RHS above.  Through this lens, the special case \eqref{KdVwp} of gKdV equations have been of great interest in the literature.

The first phase of results for \eqref{KdVwp} addressed the construction of solutions using the inverse scattering transform.  Two particularly common choices for $V$ are profiles that are periodic~\cites{Kuznetsov1974,Ermakova1982,Ermakova1982a,Firsova1988}, which describe localized defects to periodic wave trains, and step-like~\cites{Buslaev1962,Cohen1984,Cohen1987,Kappeler1986}, which arise in the study of bore propagation and rarefaction waves.  Later, the authors of~\cites{Egorova2009,Egorova2009a} established a general framework that encompasses both of these cases, and even a mixture of the two as $x\to+\infty$ and $x\to-\infty$.  Although the main thrust is to prove existence, a key step in many of these works is establishing the persistence of regularity for solutions, much like \eqref{ap est}.

While many of these results work at high regularity, these methods for existence have been adapted to classes of one-sided step-like initial data~\cites{Grudsky2014,Rybkin2011,Rybkin2018} and even to one-sided step-like elements of $H^{-1}_{\text{loc}}(\R)$~\cite{Grudsky2015}. Despite the lack of assumptions as $x\to -\infty$ (the direction in which radiation propagates), these low-regularity results require rapid decay at $x\to +\infty$ and global boundedness from below.

An alternative approach was introduced in~\cite{Menikoff1972}, which proves global existence and uniqueness for initial data that satisfies $u(0,x) = o(|x|)$ as $x\to\pm\infty$.  Here, the background wave $V$ is evolved according to the inviscid Burger's equation $\ddt V = -6VV'$ using the method of characteristics, and then solutions to the equation for $q$ are constructed using a family of discretized approximate equations.  Existence and uniqueness for $q$ is then established in a weighted $H^3$ space.

Inspired by the theory of tidal bores and rarefaction waves (as well as kink solutions to KdV with higher-power nonlinearities), authors began turning their attention to the well-posedness problem for step-like initial data.  The first phase of results employed BBM and parabolic  regularizations~\cites{Bona1994,Iorio1998,Zhidkov2001},
which involve adding a term to the LHS of \eqref{KdVwp} that eases the proof of well-posedness ($-\eps\partial_t\partial^2_xq$ and $-\eps\partial^2_xq$ for $\eps>0$, respectively) and then sending $\eps\to 0$.  The former was introduced by Bona and Smith \cite{Bona1975} in the case $V\equiv 0$ and leads to the Benjamin--Bona--Mahony equation, for which the method is named.
These approaches culminated in local well-posedness for initial data $q\in H^s$ for $s>\frac{3}{2}$, and was later advanced to $s>1$ in~\cite{Gallo2005} through the incorporation of Strichartz estimates.

Recently, the work~\cite{Palacios2023} extended local well-posedness to the range $s > \frac{1}{2}$ using a synthesis of much more modern tools for well-posedness.  In addition to the cases of step-like and periodic background waves $V$, this result also applies to more general choices of $V$ with bounded asymptotics as $x\to\pm\infty$, and even more general nonlinearities in \eqref{KdV}.  Nevertheless, it only applies to \eqref{gKdV} in the special case where $a_1\equiv 0$ and $a_2$ is constant.

Around the same time, the second author proved that \eqref{KdV} is globally well-posed for initial perturbations $q\in H^{-1}$ in~\cite{Laurens2023}, provided that the background wave $V$ is a suitable solution to \eqref{KdV}.  These conditions on $V$ do include regularity but do not impose any assumptions on spatial asymptotics.  In particular, this result applies to the important cases of smooth periodic~\cite{Laurens2023} and step-like~\cite{Laurens2022} initial data.

All of these works for \eqref{KdVwp} focus on the coefficients $a_3$ and $a_4$ in \eqref{gKdV}.  The introduction of the coefficient $a_1$ already significantly changes the analysis.  The local well-posedness of this equation for initial data in $H^s$ with $s>\frac{3}{2}$ and suitable coefficients was demonstrated in~\cite{Israwi2013} using energy methods.  Recently, local well-posedness was extended to the range $s>\frac{1}{2}$ in~\cite{Molinet2023}.

In the most general case, the optimal well-posedness results that cover the $a_2$ term in \eqref{gKdV} are~\cites{Craig1992,Akhunov2019, Ben} to the best of our knowledge.  These works establish well-posedness for initial data at a considerably high regularity, but they apply to very general families of nonlinear equations with a KdV-type dispersion.

Let us now turn to our methods for proving Theorem~\ref{t:energy}.  At the center of our analysis lies the renormalized perturbation determinant $\alpha(\kappa,u)$ (see \eqref{alpha} below for details).  The perturbation determinant is a conserved quantity for \eqref{KdV} originating from scattering theory, and is rooted in the complete integrability of this system.  However, the authors of~\cite{Killip2018} (and independently~\cite{Rybkin2010}) discovered a renormalization $\alpha$ of this quantity that satisfies
\begin{equation}
\kappa \alpha(\kappa,u) \approx \int \frac{|\wh{u}(\xi)|^2}{\xi^2 + 4\kappa^2} \dxi =: \norm{u}_{H^{-1}_\kappa}^2
\quad\text{for all }\kappa \gg \norm{u}_{H^{-1}}^2
\label{Hsk intro}
\end{equation}
and is still conserved, and used this to prove a global-in-time a-priori estimate for solutions to \eqref{KdV} in $H^s$ spaces for $-1 \leq s < 1$.  At the same time, conserved energies in $H^s$ spaces  were independently constructed in \cite{KT} for a full range  of Sobolev exponents $s \geq -1$. As it turned out,  these energies are also connected to the perturbation determinant, but with a different choice of renormalization.  See also the earlier work~\cite{B} where similar bounds were proved at slightly higher regularity $s > -\frac45$.  

The same quantity $\alpha$ was then used as a starting point in~\cite{Killip2019} to prove that \eqref{KdV} is well-posed in $H^{-1}$ using the method of commuting flows.

Of course, introducing the coefficients $a_j$ in \eqref{gKdV} breaks the conservation of $\alpha$, along with all of the other conservation laws of \eqref{KdV}.  In general, the value of $\alpha$ can grow in time, and consequently we cannot expect a global-in-time estimate.  Instead, in order to establish \eqref{ap est}, we prove that $\alpha$ must remain finite for a short period of time using a bootstrap/continuity argument.

However, we immediately encounter an obstacle.  As an illustrative example, consider the $L^2$ norm; if $u$ is a Schwartz solution of \eqref{KdV}, then the $L^2$ norm of $u(t)$ is constant in time.  On the other hand, if $u$ solves \eqref{gKdV}, then the $L^2$ norm evolves according to
\begin{equation}
\partial_t \int \tfrac{1}{2} u^2\, dx = \int -a_1 (\partial_xu)^2 + a_2 u^3 - \tfrac{1}{2}(\partial_x a_3)u^2 + a_4 u^2\, dx .
\label{L2}
\end{equation}
How can we bound the right-hand side solely in terms of $\norm{u}_{L^2}$ in order to close the argument?

Our strategy in this paper is to use local smoothing: due to the dispersive nature of the equation, we expect a gain in the regularity of solutions locally in space on average in time.  For \eqref{KdV}, this effect was discovered by Kato~\cite{Kato1983}, who proved that $L^2$ solutions are in fact in $H^1$ locally in space at almost every time.  In particular, this gives us a way to make sense of the  RHS of \eqref{L2} even for $L^2$ solutions $u$.

For the a-priori estimate \eqref{ap est}, we encounter an analogous loss of derivatives phenomenon in computing the time evolution of $\alpha$.  The local smoothing effect attendant to this problem is to show that an $H^{-1}$ solution of \eqref{gKdV} is in $L^2$ locally in space and time:
\begin{theorem}[Local smoothing estimate]\label{t:le}
There exists $\epsilon >0$ so that, if the smooth coefficients $a_j$ satisfy the bounds in \eqref{hyp 1}--\eqref{hyp 2} uniformly for $|t|\leq T$ and $x\in\R$, then for any $A>0$ there exist constants $C,T>0$ so that for any smooth solution $u$ to \eqref{gKdV} in $[-T,T]$ whose initial data satisfies
\begin{equation}
\norm{u(0)}_{H^{-1}} \leq A,
\end{equation}
satisfies the local energy bound
\begin{equation}
\sup_{x_0\in\R} \int_{-T}^T \int_{x_0-1}^{x_0+1} |u(t,x)|^2\,dx\,dt \leq C .
\label{ls est}
\end{equation}
\end{theorem}

We will prove \eqref{ap est} and \eqref{ls est} simultaneously, using a bootstrap argument in terms of both $\alpha$ and a local smoothing norm.

In the case where the coefficients $a_j\equiv 0$ vanish, an analogous local smoothing estimate was first proved in~\cite{Buckmaster2015}; see also the earlier work in \cite{B}, where similar local smoothing estimates are proved for  $H^s$ solutions with $s > -\frac45$.  However, our proof is more closely related to the alternative approach in~\cite{Killip2019}*{Th.~1.2} using the renormalized perturbation determinant.  This argument is made possible by the discovery of a density $\rho(x;\kappa,u)$ for $\alpha$ (see \eqref{alpha} below) not known during the earlier work~\cite{Killip2018}.  Moreover, if $u$ solves \eqref{KdV}, then $\rho$ satisfies the density-flux equation (or `microscopic conservation law')
\begin{equation}
\partial_t\rho + \partial_xj = 0
\label{alpha micro intro}
\end{equation}
for a certain current $j(x;\kappa,u)$ (see \eqref{j}).  Integrating this equation in space yields the conservation of $\alpha$.  Alternatively, first multiplying by a smooth step function and then integrating in space leads to a local smoothing estimate; this is basis of the short proof presented in~\cite{Killip2019}.

Turning our attention back to \eqref{gKdV}, we again encounter obstacles.  The presence of the coefficients $a_j$ breaks the conservation law \eqref{alpha micro intro}, and instead we have
\begin{equation}
\partial_t\rho + \partial_xj = d\rho|_u \big[ (a_1u')' + a_2u^2 + a_3u' + a_4u \big] .
\label{alpha micro intro 2}
\end{equation}
(Here, $d\rho|_u$ denotes the functional derivative of $\rho$ at $u$; see \eqref{drho} for details.)  All of terms on the RHS of \eqref{alpha micro intro 2} are new, and must be controlled in terms of $\alpha$ and our local smoothing norm.

Many of our estimates are in the spirit of~\cite{Bringmann2021}, which proves that the fifth-order KdV equation is globally well-posed in $H^{-1}$.  In order to prove their result, the authors had to prove an analogous local smoothing estimate for $H^{-1}$ solutions of their system and a family of approximate equations.  Together, the estimates in~\cites{Killip2019,Bringmann2021} provide upper bounds on the density $\rho$ and current $j$ in terms of the local smoothing and $H^{-1}$ norms, the latter of which can then be controlled by $\alpha$ using \eqref{Hsk intro}.  For example, when $\kappa\geq 1$ we may bound
\begin{equation}
\norm{u}_{H^{-1}_\kappa}^2 = \int \frac{|\wh{u}(\xi)|^2}{\xi^2 + 4\kappa^2} \dxi \leq \norm{u}_{H^{-1}}^2
\label{Hsk intro 2}
\end{equation}
and use this to construct $\alpha$ for all $\norm{u}^2_{H^{-1}} \ll 1$.  For the \eqref{KdV} and fifth-order KdV equations, this small-data assumption does not pose a problem; one can use the scaling symmetry to make the initial data arbitrarily small, and then the exact conservation of $\alpha$ implies that solutions remain small globally in time.

By comparison, \eqref{gKdV} does not posses a scaling symmetry in general, and we expect that the $H^{-1}$ norm of solutions can be growing in time without bound.  This in turn forces us to choose $\kappa \gg \norm{u}^2_{H^{-1}}$ large, which introduces a large implicit constant in \eqref{Hsk intro 2}.  Consequently, the estimates in~\cites{Killip2019,Bringmann2021} are insufficient, even in estimating the LHS of~\eqref{alpha micro intro 2}.  In order to close our bootstrap argument, we must establish new estimates for the density $\rho$ and current $j$ using the $H^{-1}_\kappa$ norm, rather than $H^{-1}$, as this is the quantity whose growth we can efficiently estimate using~\eqref{Hsk intro}. In comparison to the estimate \eqref{Hsk intro 2} used throughout~\cites{Killip2019,Bringmann2021}, this requires a much more careful estimation of high frequency contributions.

This paper is organized as follows.  We begin in Section~\ref{s:prelim} by reviewing the material developed in \cites{Killip2019,Bringmann2021} that we will need, including the key players $\alpha$, $\rho$, and $j$ mentioned previously as well as the estimates they satisfy.  We then proceed in Section~\ref{s:lsk} by introducing our local smoothing norm and developing the new estimates required for our bootstrap argument.

We employ these ingredients in Sections~\ref{s:ls est} and \ref{s:ap est} to prove more precise versions of our local smoothing and a-priori estimates, Theorem~\ref{t:local smoothing} and Proposition~\ref{t:ap est}, respectively.  We then combine these two estimates with a bootstrap argument, and present the full result in Theorem~\ref{t:ap est full}.  Finally, we conclude Section~\ref{s:ap est} by describing how the statements of Theorems~\ref{t:energy} and \ref{t:le} and Corollary~\ref{t:KdV vb} represent special cases of Theorem~\ref{t:ap est full}.

\subsection{Acknowledgments}
The authors are grateful to the referees for their careful reading and valuable suggestions, which greatly improved the exposition of the paper.

While working on this project, the first author was supported by NSF Grant DMS-2348908, the Miller Foundation, a Simons Fellowship, and a Vilas Associate Fellowship. The second author was partially supported by a Vilas Associate Fellowship.

\section{Preliminaries}
\label{s:prelim}

Here we recall some of the key objects from~\cite{Killip2019} which were used to prove that~\eqref{KdV} is well-posed in $H^{-1}$, and which will serve us well in obtaining the a-priori bounds  for  our nonintegrable model \eqref{gKdV} displayed in Theorem~\ref{t:energy} and Theorem~\ref{t:le}.  Our starting point is the self-adjoint Lax operator associated to KdV, 
\[
L_u := -\partial_x^2 + u.
\]
Formally, the spectrum of $L_u$ is conserved by the KdV flow, because $L_{u(t)}$ and $L_{u(s)}$ are unitarily equivalent operators. In general $L_u$
is not invertible, which is one reason it is convenient to replace it with $L_u + \kappa^2$; given $u \in H^{-1}$ the operator $L_u+\kappa^2$ is invertible if $\kappa>0$ is large enough, as we will discuss more thoroughly below.  
One may construct conserved quantities for KdV by looking 
at the trace of functions of $L_u+\kappa^2$, and in particular the trace of its renormalized logarithm,
for any $\kappa$ sufficiently large.  As shown in \cite{Killip2019}, this trace is closely related to the diagonal of the kernel of $L_u +\kappa^2$; the first step in the analysis is to construct the diagonal of the integral kernel $G(x,y)$ for the resolvent 
\begin{equation}
\label{def:resolvant}
R_u(\kappa):=(-\partial^2_x + u + \kappa^2)^{-1}
\end{equation}
of the Lax operator associated to a state $u$. With this operator in mind, it will be convenient to work with the norms
\begin{equation*}
\norm{f}_{H^s_\kappa}^2 := \int |\wh{f}(\xi)|^2 (\xi^2+4\kappa^2)^s\dxi ,
\end{equation*}
where our convention for the Fourier transform is
\begin{equation*}
\wh{f}(\xi) = \frac{1}{\sqrt{2\pi}} \int e^{-i\xi x} f(x)\dx \quad\text{so that}\quad \langle f,g\rangle = \int \ol{f}(x)g(x)\dx = \langle \wh{f},\wh{g} \rangle .
\end{equation*}

In the definition of the $H^s_\kappa$-norms, the constant $4$ is simply included in order to make \eqref{HS} an equality.  One should see these spaces as versions of the traditional inhomogeneous Sobolev spaces $H^s$ but adapted to the frequency scale $\kappa$ rather than the unit frequency. Topologically these are equivalent to the classical Sobolev spaces $H^s$, but with the implicit constants in the norm equivalence depending on $\kappa$. In particular,  the following elementary estimates hold
\begin{equation*}
\norm{fg}_{H^{\pm 1}_\kappa} \lesssim \norm{f}_{H^1} \norm{g}_{H^{\pm 1}_\kappa}
\quad\text{and}\quad
\norm{fg}_{H^{\pm 1}_\kappa} \lesssim \norm{f}_{W^{1,\infty}} \norm{g}_{H^{\pm 1}_\kappa}
\end{equation*}
uniformly for $\kappa\geq 1$.  Notice that the results for $H^{-1}_\kappa$ follow from those for $H^1_\kappa$ by duality.  We will also frequently use the embedding
\begin{equation*}
\norm{f}_{L^\infty} \lesssim \kappa^{-\frac12} \norm{f}_{H^1_\kappa} \quad\text{uniformly for }\kappa\geq 1 ,
\end{equation*}
which implies the algebra bound
\begin{equation*}
\norm{fg}_{H^1_\kappa} \lesssim \kappa^{-\frac12} \norm{f}_{H^1_\kappa}\norm{g}_{H^1_\kappa} \quad\text{uniformly for }\kappa\geq 1 .
\end{equation*}

The resolvent associated to $u\equiv 0$,
\begin{equation}
R_0(\kappa) := (-\partial^2_x + \kappa^2)^{-1},
\label{R0}
\end{equation}
plays an important role in what follows, and in particular has the mapping property
\[
R_0: H^s_{\kappa} \rightarrow H^{s+2}_{\kappa}, \quad s\in \mathbb{R}.
\]
This property is also shared by $R_u$ for $u\in H^{-1}$ for $\kappa$ large enough, but only for restricted range of $s$. Indeed, one can formally use $R_0$ to obtain an expansion for $R_u$:
\begin{equation}
\label{Ru-exp}
R_u=\sum_{\ell=0}^{\infty} (-1)^\ell (R_0u)^\ell R_0,
\end{equation}
which is justified for $u\in H^{-1}$ for $\kappa$ large enough.  Given $u\in H^{-1}$, this series converges to a bounded operator from $H^{-1}_\kappa$ into $H^1_\kappa$ for all $\kappa$ sufficiently large by the following lemma.

\begin{lemma}[\cite{Killip2018}]
For $\kappa>0$, we have
\begin{equation}
\bnorm{ \sqrt{R_0} f \sqrt{R_0} }_{\hs} = \kappa^{-\frac{1}{2}} \norm{ f}_{H^{-1}_\kappa} .
\label{HS}
\end{equation}
\end{lemma}

Here, $\hs$ denotes the class of Hilbert--Schmidt operators on $L^2(\R)$.  Such operators are automatically continuous due to the inequality
\begin{equation*}
\norm{A}_{\op} \leq \norm{A}_{\hs} = \sqrt{ \tr\{ A^*A \} } ,
\end{equation*}
where $\op$ denotes the operator norm.  Moreover, the space $\hs$ forms a two-sided ideal in the space of bounded operators:
\begin{equation*}
\norm{ABC}_{\hs} \leq \norm{A}_{\op} \norm{B}_{\hs} \norm{C}_{\op} .
\end{equation*}
Lastly, the product of two Hilbert--Schmidt operators is trace class, and we have
\begin{equation*}
|\tr(AB)| \leq \norm{AB}_{\tc} \leq \norm{A}_{\hs} \norm{B}_{\hs} .
\end{equation*}
Here, $\tc$ denotes the trace class; this consists of operators whose singular values are $\ell^1$-summable.  In this way, the second inequality above is simply an application of the Cauchy--Schwarz inequality.

The Green's function is the integral kernel for the resolvent $R_u$, which may be expressed as
\[
G(x,y) = \langle \delta_x, (-\partial_x^2 + u + \kappa^2)^{-1}\delta_y \rangle.
\]
As mentioned previously, its restriction to the diagonal plays a key role in our analysis.
\begin{proposition}[Diagonal Green's function \cite{Killip2019}]
\label{t:g}
There exists a constant $C>0$ so that the following statements are true for any $R>0$:
\begin{enumerate}
\item For each $\norm{u}_{H^{-1}_\kappa} \leq R$, the diagonal Green's function
\begin{equation*}
g(x;\kappa,u) := \langle \del_x , ( -\partial^2_x + u + \kappa^2 )^{-1} \del_x \rangle
\end{equation*}
exists for all $x\in\R$ and $\kappa\geq 1 + CR^2$.
\item The mappings
\begin{equation}
u \mapsto g - \tfrac{1}{2\kappa} \quad\tx{and}\quad
u\mapsto \tfrac{1}{g} - 2\kappa
\label{g diffeo}
\end{equation}
are real-analytic functionals from $\{ u: \norm{u}_{H^{-1}} \leq R \}$ into $H^1$ for all $\kappa\geq 1 + CR^2$.
\item We have the estimates
\begin{align}
\norm{ g(\kappa,u) - \tfrac{1}{2\kappa} }_{H^1_\kappa} &\lesssim \kappa^{-1} \norm{u}_{H^{-1}_\kappa} ,
\label{g est} \\
\bnorm{ \tfrac{1}{g(\kappa,u)} - 2\kappa }_{H^1_\kappa} &\lesssim \kappa \norm{u}_{H^{-1}_\kappa}
\label{1/g est}
\end{align}
uniformly for $\kappa\geq 1 + CR^2$.
\end{enumerate}
\end{proposition}

As a consequence of the expansion of $R_u$ in \eqref{Ru-exp}, we have a similar expansion for $g$:
\begin{equation}
g = \tfrac{1}{2\kappa} + h_1 + h_2 + \dots, \qquad
h_\ell(x) = (-1)^\ell \langle \del_x , (R_0u)^\ell R_0 \del_x \rangle .
\label{g}
\end{equation}
In particular, using the integral kernel $\langle \del_x,R_0(\kappa)\del_y\rangle = \tfrac{1}{2\kappa} e^{-\kappa|x-y|}$ for the free resolvent, we find that 
\[
h_1 = -\kappa^{-1}R_0(2\kappa)u.
\]

As it turns out, it is the function $\frac{1}{g}$ that is most closely related to the trace of the renormalized logarithm mentioned previously. 
However, in general, $\int \frac{1}{g}-2\kappa\,dx$ diverges for $u$ in any $H^s$ space. To rectify this, one also needs to remove the linear term in $u$, which formally integrates to 
a multiple of $\int u \, dx$; this is a conserved quantity for the KdV flow, and thus is harmless. As a side note, a similar renormalization was also implemented in \cite{KT} for the logarithm of the transmission coefficient. 

Taking these renormalizations into account, we are now prepared to define the conserved quantity which controls the $H^{-1}$ norm of $u$:
\begin{equation}
\alpha(\kappa,u) := \int \rho \dx , \qquad
\rho(x;\kappa,u) := - \frac{1}{2g(x;\kappa,u)} + \kappa + 2\kappa R_0(2\kappa) u .
\label{alpha}
\end{equation}
This formula for $\alpha$ is the trace of the integral kernel $-1/2G(x,y;\kappa,u)$ restricted to the diagonal, with the first two terms of its Taylor series about $u \equiv 0$ removed.
\begin{proposition}[Introducing $\alpha$ \cite{Killip2019}]
\label{t:alpha}
There exists a constant $C>0$ so that the following statements are true for any $R>0$:
\begin{enumerate}
\item The quantities $\rho(x;\kappa,u)$ and $\alpha(\kappa,u)$ defined in~\eqref{alpha} are finite and nonnegative for all $x\in\R$, $\norm{u}_{H^{-1}_\kappa}\leq R$, and $\kappa\geq 1+CR^2$.
\item The mapping $u\mapsto \alpha$ is a real analytic functional on $\{ u: \norm{u}_{H^{-1}_\kappa}\leq R \}$ for all $\kappa\geq 1+CR^2$.
\item We have
\begin{equation}
\tfrac{1}{4} \kappa^{-1} \norm{u}_{H^{-1}_\kappa}^2 \leq \alpha(\kappa,u) \leq \kappa^{-1} \norm{u}_{H^{-1}_\kappa}^2
\label{alpha est}
\end{equation}
for all $\norm{u}_{H^{-1}_\kappa} \leq R$ and $\kappa\geq 1+CR^2$.
\end{enumerate}
\end{proposition}

The authors of \cite{Killip2019} also proved that under the \eqref{KdV} flow, we have the following density-flux equation 
\begin{equation}
\ddt \rho + j' = 0 ,
\label{alpha micro}
\end{equation}
where $\rho = \rho(x;\kappa,u)$ is defined in~\eqref{alpha} and
\begin{equation}
j(x;\kappa,u) := \tfrac{1}{g} ( 4\kappa^3 g - 2\kappa^2 + u ) + 2\kappa R_0(2\kappa) (u''-3u^2) .
\label{j}
\end{equation}
We see in particular by integrating \eqref{alpha micro} in space that $\alpha$ is conserved under the \eqref{KdV} flow.

By comparison, the coefficients $a_j$ in \eqref{gKdV} break the density-flux equation \eqref{alpha micro}.  Instead, we obtain an additional term as follows:
\begin{equation}
\ddt \rho + j' = d\rho|_u \big[ (a_1u')' + a_2u^2 + a_3u' + a_4u \big] ,
\label{alpha micro 2}
\end{equation}
where 
\begin{equation}
d\rho|_u(f)
= \frac{d}{ds} \rho(x;\kappa,u+sf) \bigg|_{s=0}
= \tfrac{1}{2g^2} dg|_u (f) + 2\kappa R_0(2\kappa)f
\label{drho}
\end{equation}
is the functional derivative of $\rho$ at $u$, and
\begin{equation}
dg|_u(f) = - \int G(x,y) f(y) G(y,x) \dy .
\label{dg}
\end{equation}
The identity \eqref{alpha micro 2} lies at the heart of both of the proofs of our a-priori and local smoothing estimates.

In order to leverage local smoothing, we will frequently need to commute slowly varying weights past $R_0$, as in the following estimates.
\begin{lemma}
If $w:\R \to (0,\infty)$ satisfies
\begin{equation}
|w'(x)| + |w''(x)| \lesssim w(x) \quad\tx{and}\quad 
\frac{w(y)}{w(x)} \lesssim e^{|x-y|/2} 
\label{w}
\end{equation}
uniformly for $x,y\in\R$, then
\begin{equation}
\norm{ w R_0 \tfrac{1}{w} }_{L^p\to L^p} \lesssim \kappa^{-2} , \quad
\norm{ w \partial R_0 \tfrac{1}{w} }_{L^p\to L^p} \lesssim \kappa^{-1} , \quad
\norm{ w R_0 \tfrac{1}{w} }_{H^{-1}_\kappa\to H^1_\kappa} \lesssim 1 
\label{comm}
\end{equation}
uniformly for $\kappa \geq 1$, for any $1\leq p \leq \infty$.
\end{lemma}
The proof is elementary; see, for example,~\cite{Bringmann2021}*{Lem.~2.7}. 

We will use this lemma in particular for the functions $\psi^\ell$, $\ell=1,2,3$, where 
\begin{equation*}
\psi(x) =\operatorname{sech}(\tfrac{x}{6}).
\end{equation*}
The constant 6 is simply included to ensure that $\psi^{\ell}$ satisfies \eqref{w} for $\ell=1,2,3$.  (In turn, the constant 2 in \eqref{w} is based on the integral kernel $\langle \del_x,R_0\del_y\rangle = \tfrac{1}{2\kappa} e^{-\kappa|x-y|}$ and the condition $\kappa\geq 1$.)

Over the course of our analysis, we will also encounter the `double commutator' 
\begin{equation*}
R_0\psi^2 R_0 - \psi R_0^2\psi =[[R_0, \psi], R_0\psi].
\end{equation*}
In order to bound this operator, the naive estimate~\eqref{comm} unfortunately does not yield enough decay as $\kappa\to\infty$.  To this end, we record the following operator estimates, which capture the double commutator structure:
\begin{lemma}
We have
\begin{gather}
\bnorm{ \tfrac{1}{\psi} ( R_0\psi^2 R_0 - \psi R_0^2\psi ) \tfrac{1}{\psi} }_{H^{-2}_\kappa \to H^2_\kappa} \lesssim \kappa^{-2} ,
\label{double comm} \\
\bnorm{ \tfrac{1}{\psi} ( R_0 \partial \psi^2 R_0\partial - \psi R_0^2 \partial^2 \psi ) \tfrac{1}{\psi} }_{H^{-2}_\kappa \to H^2_\kappa} \lesssim 1
\label{double comm1}
\end{gather}
uniformly for $\kappa\geq 1$.
\end{lemma}

In comparison to \eqref{comm}, the proofs of \eqref{double comm} and \eqref{double comm1} exhibit extra cancellation by computing the double commutator in a symmetric way.  For example, \eqref{double comm} follows from~\cite{Bringmann2021}*{Lem.~2.10}, and \eqref{double comm1} can be verified using a similar argument.

\section{The local smoothing norm}
\label{s:lsk}

Here, we will record some useful estimates for the local smoothing norm
\begin{equation}
\norm{u}_{\ls_\kappa} = \sup_{x_0\in\R} \bnorm{ \psi_{x_0}u' }_{L^2_tH^{-1}_\kappa([-T,T]\times\R)} ,
\label{LSk def}
\end{equation}
where $T>0$ and
\begin{equation*}
\psi(x) = \operatorname{sech}(\tfrac{x}{6}) , \qquad
\psi_{x_0}(x) = \psi(x-x_0) .
\end{equation*}

We begin with the following elementary observations:
\begin{lemma}
Given a Schwartz function $\phi(x)$, we have
\begin{align}
\norm{\phi u}_{L^2_tH^{-2}_\kappa}
&\lesssim \kappa^{-1} \norm{u}_{L^2_tH^{-1}_\kappa}  ,
\label{LSk 4 new}\\
\norm{(\phi u)'}_{L^2_tH^{-2}_\kappa} + \norm{ \phi u'}_{L^2_tH^{-2}_\kappa} 
&\lesssim \kappa^{-1} \big( \norm{u}_{\ls_\kappa} + \norm{u}_{L^2_tH^{-1}_\kappa} \big),
\label{LSk 3 new}\\
\norm{(\phi u)''}_{L^2_tH^{-2}_\kappa} + \norm{ \phi u'' }_{L^2_tH^{-2}_\kappa}
&\lesssim \norm{u}_{\ls_\kappa} + \kappa^{-1} \norm{u}_{L^2_tH^{-1}_\kappa} ,
\label{LSk 2 new}
\end{align}
uniformly for $\kappa\geq 1$ and $T>0$.  (All space-time norms are taken over $[-T,T]\times\R$.)
\end{lemma}

\begin{proof}
For the estimate \eqref{LSk 4 new} we simply note that
\[
\| \phi u\|_{L^2_tH^{-2}_\kappa} \lesssim \| u \|_{L^2_tH^{-2}_\kappa} \lesssim \kappa^{-1} \norm{u}_{L^2_tH^{-1}_\kappa} .
\]

For the estimate \eqref{LSk 3 new} we use the Leibnitz rule 
$(\phi u)' = \phi'u + \phi u'$, and apply~\eqref{LSk 4 new} to the contribution of the first term $\phi'u$.
For the second, we write
\begin{equation*}
\int \psi_{z}^2(x) \dz = c
\end{equation*}
for some constant $c$, so that
\begin{equation*}
\phi(x) = \tfrac{1}{c} \int \phi(x)\psi_{z}^2(x) \dz .
\end{equation*}
Then, it follows that 
\begin{equation*}
\norm{\phi u'}_{H^{-1}_\kappa}
\leq \tfrac{1}{c} \int \norm{\phi \psi_z^2 u'}_{H^{-1}_\kappa}\, \dz  \lesssim  \int \norm{\phi \psi_z}_{H^4} \| \psi_z u'\|_{H^{-1}_\kappa} \, \dz. 
\end{equation*}
Integrating in $z$ and noting that
\begin{equation*}
\int \norm{\phi\psi_z}_{H^4} \dz \lesssim_\phi 1 ,
\end{equation*}
we obtain 
\begin{equation}\label{phi-u'}
\| \phi u'\|_{H^{-1}_\kappa} \lesssim \|u\|_{\ls_\kappa} .
\end{equation}
The estimate~\eqref{LSk 3 new} then follows, as 
\[
\| \phi u'\|_{H^{-2}_\kappa} \lesssim \kappa^{-1} \| \phi u'\|_{H^{-1}_\kappa}.
\]

For the estimate \eqref{LSk 2 new}, we simply write 
\[
(\phi u)'' = (\phi u')' + \phi' u' +\phi'' u,
\]
and use \eqref{phi-u'} twice
and \eqref{LSk 4 new} once.
\end{proof}

The operator estimate~\eqref{LSk H1k 2} below will play an important role in our analysis.  In comparison to~\eqref{HS}, we now bound the operator norm (rather than Hilbert--Schmidt) and only require a local norm of $u$ to do so.

\begin{lemma}
We have
\begin{equation}
\norm{ fu' }_{L^2_tH^{-1}_\kappa} \lesssim \kappa^{-\frac12} \norm{f}_{H^1_\kappa} \norm{u}_{\ls_\kappa} \quad\tx{for all }f\in H^1(\R),
\label{LSk H1k}
\end{equation}
uniformly for $\kappa \geq 1$.  As a consequence,
\begin{equation}
\bnorm{\sqrt{R_0} u' \sqrt{R_0}}_{L^2_t\op} \lesssim \kappa^{-\frac12}\norm{u}_{\ls_\kappa} .
\label{LSk H1k 2}
\end{equation}
\end{lemma}
\begin{proof}
We argue by duality.  Let $h\in L^2_tH^1_\kappa$.  Fix a smooth partition of unity
\begin{equation*}
\sum_{j\in\Z} \chi_j^2 \equiv 1 ,
\quad\tx{with}\quad
\chi_j \equiv 1 \tx{ on } [j,j+1]
\quad\tx{and}\quad
\operatorname{supp} \chi_j \subset [j-1,j+2] .
\end{equation*}
Then
\begin{equation*}
\langle h , fu' \rangle 
= \sum_{j\in\Z} \int \wt{\chi}_j^2 u'\cdot \chi_j^2 f\ol{h} \dx ,
\end{equation*}
where $\wt{\chi}_j^2 = \chi_{j-1}^2 + \chi_j^2 + \chi_{j+1}^2$ is a fattened bump function.  We estimate
\begin{align*}
\int_{-T}^T |\langle h , fu' \rangle|\dt 
&\lesssim \sum_j \snorm{\wt{\chi}_j^2 u'}_{L^2_tH^{-1}_\kappa} \norm{ \chi_j^2 f\ol{h} }_{L^2_tH^1_\kappa} \\
&
\lesssim \kappa^{-\frac12} \sum_j \norm{u}_{\ls_\kappa} \norm{ \chi_j f }_{H^1_\kappa} \norm{ \chi_j h }_{L^2_tH^1_\kappa}
\\
&\lesssim \kappa^{-\frac12} \norm{u}_{\ls_\kappa} \bigg( \sum_j \norm{\chi_j f}_{H^1_\kappa}^2 \bigg)^{\frac{1}{2}} \bigg( \sum_j \norm{\chi_j h}_{L^2_tH^{1}_\kappa}^2 \bigg)^{\frac{1}{2}} \\
&\lesssim \kappa^{-\frac12}\norm{u}_{\ls_\kappa} \norm{f}_{H^1_\kappa} \norm{h}_{L^2_tH^1_\kappa} .
\end{align*}
Taking a supremum over $\norm{h}_{L^2_tH^1_\kappa}\leq 1$, the estimate \eqref{LSk H1k} follows.

The second estimate \eqref{LSk H1k 2} now follows immediately:
\begin{equation*}
\bnorm{\sqrt{R_0} u' \sqrt{R_0}}_{L^2_t\op} \lesssim \kappa^{-\frac12} \bnorm{\sqrt{R_0}}_{H^{-1}_\kappa \to L^2} \norm{u}_{\ls_\kappa} \bnorm{\sqrt{R_0}}_{L^2\to H^1_\kappa}  
\lesssim \kappa^{-\frac12}\norm{ u}_{\ls_\kappa} . \qedhere
\end{equation*}
\end{proof}

Our next lemma provides an estimate for the diagonal Green's function analogous to \eqref{g est} that captures the gain of regularity we expect from our local smoothing norm.

\begin{lemma}
 There exists a constant $C>0$ so that for any
Schwartz function $\phi(x)$, we have
\begin{equation}
\norm{\phi g'}_{L^2_tH^1_\kappa} \lesssim_\phi \kappa^{-1} \norm{u}_{\ls_\kappa}
\label{LSk g}
\end{equation}
uniformly for 
\begin{equation}
\kappa \geq 1 + C\norm{u}_{L^\infty_tH^{-1}_\kappa}^2 .
\label{k0}
\end{equation}
\end{lemma}
\begin{proof}
Using a `continuous partition of unity' as in the proof of \eqref{LSk 3 new}, it suffices to prove \eqref{LSk g} in the special case where $\phi = \psi$.  Again, we will argue by duality.  Given $f\in L^2_tH^{-1}_\kappa$, the expansion \eqref{g} for $g$ yields
\begin{equation*}
\int_{-T}^T \bigg| \int f\psi g' \dx\bigg| \dt
\leq \sum_{\ell\geq 1} \int_{-T}^T \big| \tr\big\{ f\psi [\partial, (R_0u)^\ell R_0] \big\} \big| \dt .
\end{equation*}

When $\ell=1$, there is only one copy of $u$ on which the derivative $\partial$ may fall.  In this case, we write
\begin{equation*}
f \psi R_0 u' = f [\psi R_0 \tfrac{1}{\psi}] \psi u'
\end{equation*}
and note that the operator in square brackets is bounded $H^{-1}_\kappa \to H^1_\kappa$ by \eqref{comm}.  Therefore
\begin{equation*}
\big| \tr\big\{ f\psi R_0u' R_0 \big\} \big|
\lesssim \bnorm{ \sqrt{R_0} f\sqrt{R_0} }_{\hs} \bnorm{\sqrt{R_0} \psi u' \sqrt{R_0} }_{\hs}
\lesssim \kappa^{-1} \norm{f}_{H^{-1}_\kappa} \norm{ \psi u' }_{H^{-1}_\kappa} ,
\end{equation*}
and thus
\begin{equation*}
\int_{-T}^T  \big| \tr\big\{ f\psi R_0u' R_0 \big\} \big|
\dt
\lesssim \kappa^{-1}\norm{f}_{L^2_tH^{-1}_\kappa} \norm{u}_{\ls_\kappa} .
\end{equation*}

Next, we turn to the case where $\ell\geq 2$.  We distribute the derivative $\partial$ using the product rule $[\partial, AB] = [\partial,A]B + A [\partial, B]$.  This produces one copy of $u'$, and we estimate its contribution in operator norm using~\eqref{LSk H1k 2}.  The requirement $\ell\geq 2$ then guarantees that there is at least one other copy of $u$, which together with $f$ provides us with two terms that can be put in Hilbert--Schmidt norm:
\begin{align*}
&\sum_{\ell\geq 2} \int_{-T}^T \big| \tr\big\{ f\psi [\partial, (R_0u)^\ell R_0] \big\} \big| \dt \\
&\qquad \leq \sum_{\ell\geq 2} \ell \bnorm{ \sqrt{R_0} f\psi \sqrt{R_0} }_{L^2_t\hs} \bnorm{ \sqrt{R_0} u' \sqrt{R_0} }_{L^2_t\op} \bnorm{ \sqrt{R_0} u \sqrt{R_0} }_{L^\infty_t\hs}^{\ell-1} \\
&\qquad\lesssim \sum_{\ell\geq 2} \ell \kappa^{-1} \norm{ f }_{L^2_tH^{-1}_\kappa} \norm{ u}_{\ls_\kappa} \big( \kappa^{-\frac{1}{2}} \norm{u}_{L^\infty_tH^{-1}_\kappa} \big)^{\ell-1} \\
&\qquad\lesssim \kappa^{-1} \norm{ f }_{L^2_tH^{-1}_\kappa} \norm{u}_{\ls_\kappa} .
\end{align*}

Altogether, we conclude
\begin{equation*}
\int_{-T}^T \bigg| \int f\psi g' \dx\bigg| \dt
\lesssim \kappa^{-1} \norm{ f }_{L^2_tH^{-1}_\kappa} \norm{u}_{\ls_\kappa}.
\end{equation*}
Taking a supremum over $\norm{f}_{L^2_tH^{-1}_\kappa} \leq 1$, the estimate \eqref{LSk g} follows.
\end{proof}

The next few results provide estimates for the terms $h_\ell$ appearing in the series \eqref{g} for $g$.  We will start with $h_1$, which follows easily from \eqref{LSk 4 new}--\eqref{LSk 2 new}.
\begin{lemma}
We have
\begin{align}
\norm{\psi h_1}_{L^2_{t,x}} &\lesssim \kappa^{-2}  \norm{u}_{L^2_tH^{-1}_\kappa}  , 
\label{h1 1} \\
\norm{\psi h'_1}_{L^2_{t,x}} &\lesssim \kappa^{-2} \norm{u}_{\ls_\kappa} ,
\label{h1 2} \\
\norm{\psi h''_1}_{L^2_{t,x}} &\lesssim \kappa^{-1} \norm{u}_{\ls_\kappa} + \kappa^{-2} \norm{u}_{L^2_tH^{-1}_\kappa} 
\label{h1 3}
\end{align}
uniformly for $\kappa\geq 1$ and $T>0$.
\end{lemma}
\begin{proof}
Recall that $h_1 = -\kappa^{-1}R_0(2\kappa) u$.  Using \eqref{comm}, we see that
\begin{equation*}
\norm{ \psi h'_1 }_{L^2_{t,x}}
\lesssim \kappa^{-1} \norm{\psi h'_1}_{L^2_tH^1_\kappa}
= \kappa^{-2} \bnorm{ \psi R_0(2\kappa)\tfrac{1}{\psi} \cdot \psi u' }_{L^2_tH^1_\kappa}
\lesssim \kappa^{-2} \norm{u}_{\ls_\kappa} .
\end{equation*}
This proves \eqref{h1 2}.  For \eqref{h1 1} and \eqref{h1 3}, we also use \eqref{LSk 2 new} and \eqref{LSk 4 new}:
\begin{gather*}
\norm{\psi h_1}_{L^2_{t,x}}
\lesssim \kappa^{-1} \norm{\psi u}_{L^2_tH^{-2}_\kappa}
\lesssim \kappa^{-2}  \norm{u}_{L^2_tH^{-1}_\kappa} , \\
\norm{\psi h''_1}_{L^2_{t,x}}
\lesssim \kappa^{-1} \norm{\psi u''}_{L^2_tH^{-2}_\kappa}
\lesssim \kappa^{-1}  \big( \norm{u}_{\ls_\kappa} + \kappa^{-1} \norm{u}_{L^2_tH^{-1}_\kappa} \big)  .
\qedhere
\end{gather*}
\end{proof}

The following lemma will be useful in estimating all of the cubic and higher order terms of $j$ (see \eqref{j3 3} for details).  The proof is rather involved, because we need to be efficient in our estimation.  Specifically, it will be important that the factor $\kappa^{-\frac{11}2}$ on the RHS of \eqref{h3} is $o(\kappa^{-5})$.  (Otherwise, we would not be able to start at $\ell=3$ in \eqref{j3 3}.)
\begin{lemma}
Given a Schwartz function $\phi(x)$, we have
\begin{equation}
\sum_{\ell\geq 3} \norm{\phi h_\ell}_{L^1_{t,x}} \lesssim  (T \kappa^2)^{\frac34} \kappa^{-\frac{11}2}  \norm{u}_{L^\infty_tH^{-1}_\kappa} \big( \norm{u}_{L^\infty_tH^{-1}_\kappa}^2 + \norm{u}_{\ls_\kappa}^2 \big) 
\label{h3}
\end{equation}
uniformly for $T \leq \kappa^{-2}$ and $\kappa$ satisfying~\eqref{k0}. 
\end{lemma}

\begin{proof}
Using a continuous partition of unity argument as in the proof of \eqref{LSk 3 new}, it suffices to prove \eqref{h3} in the special case when $\phi = \psi^3$.
Arguing by duality, consider $f\in L^\infty([-T,T]\times\R)$ and write
\begin{equation}
\int_{-T}^T \int_{-\infty}^\infty f\psi^3 h_\ell\dx\dt
= (-1)^\ell \int_{-T}^T \tr\big\{ \sqrt{R_0} f\psi^3 (R_0u)^\ell \sqrt{R_0} \big\} \dt .
\label{hl}
\end{equation}

Let us begin with the $\ell = 3$ term, and expand
\begin{equation*}
\sqrt{R_0} f \psi^3 (R_0u)^3 \sqrt{R_0}  =  \sqrt{R_0} f [\psi^3R_0\tfrac{1}{\psi^3}] \psi u [\psi^2R_0\tfrac{1}{\psi^2}] \psi u [\psi R_0\tfrac{1}{\psi}]\psi u \sqrt{R_0}.
\end{equation*}
By \eqref{comm}, each operator in square brackets above is bounded $H^{-1}_\kappa\to H^1_\kappa$.  This allows us to write
\begin{equation*}
\sqrt{R_0} f\psi^3 (R_0u)^3 \sqrt{R_0} = (\sqrt{R_0} f \sqrt{R_0})A_3(\sqrt{R_0} \psi u \sqrt{R_0}) A_2(\sqrt{R_0} \psi u \sqrt{R_0}) A_1(\sqrt{R_0}\psi u \sqrt{R_0})
\end{equation*}
for operators $A_j$ with $\norm{A_j}_{\op} \lesssim 1$.  We seek to estimate the $\tc$-norm of this entire operator, for which it suffices to bound two of the factors in round brackets in $\op$ and two in $\hs$. The factor containing $f$ is directly estimated in $\op$ by 
\begin{equation}\label{op-bd}
\| \sqrt{R_0} f \sqrt{R_0}\|_{\op} 
\lesssim \kappa^{-2} \|f\|_{L^\infty}.
\end{equation}
It remains to consider the three factors that contain $u$.  Two of these factors will be estimated in the $\hs$-norm, using \eqref{HS}:
\begin{equation}\label{hs-bd}
\bnorm{ \sqrt{R_0} w \sqrt{R_0} }_{\hs} 
= \kappa^{-\frac12} \snorm{w}_{H^{-1}_\kappa}.
\end{equation}
The remaining factor will be bounded in the $\op$-norm. Employing
\eqref{op-bd} for low frequencies and \eqref{hs-bd} for high frequencies, we find
\[
\bnorm{ \sqrt{R_0} w \sqrt{R_0} }_{\op} 
\lesssim \kappa^{-\frac12} \snorm{w_{>\kappa}}_{H^{-1}_\kappa}
+ \kappa^{-2} \| w_{<\kappa}\|_{L^\infty} \lesssim  \kappa^{-\frac12} \snorm{w_{>\kappa}}_{H^{-1}_\kappa}
+ \kappa^{-2} \| w_{<\kappa}\|_{L^2}^\frac12 \| w'_{<\kappa} \|^\frac12_{L^2},
\]
which can be shortened to
\begin{equation}
\label{infty-better}
\bnorm{ \sqrt{R_0} w \sqrt{R_0} }_{\op} 
\lesssim  \kappa^{-\frac12} \snorm{w}_{H^{-1}_\kappa}^{\frac12} \snorm{w'}_{H^{-2}_\kappa}^{\frac12} .
\end{equation}
Applying these bounds to $h = \psi u$, we obtain
\[
\begin{aligned}
\bnorm{ \sqrt{R_0}f\psi^3 (R_0u)^3 \sqrt{R_0} }_{\tc} \lesssim & \
\kappa^{-2-\frac32} \snorm{h}_{H^{-1}_\kappa}^\frac52 
 \snorm{h'}_{H^{-2}_\kappa}^\frac12 
  \|f\|_{L^\infty}
 \\
 \lesssim  & \ \kappa^{-2-\frac32} \snorm{h}_{H^{-1}_\kappa}
 \big( 
c^3 \snorm{h'}_{H^{-2}_\kappa}^2 
 + c^{-1} \snorm{h}_{H^{-1}_\kappa}^2 \big)  \|f\|_{L^\infty}
 \end{aligned}
\]
for an arbitrary constant $c>0$ to be chosen later.
Integrating in time, gives
\[
\bnorm{ \sqrt{R_0}f\psi^3 (R_0u)^3 \sqrt{R_0} }_{L^1_t\tc} \lesssim \kappa^{-2-\frac32}
\snorm{h}_{L^\infty_t H^{-1}_\kappa}
 \big( 
c^3 \snorm{h'}_{L^2_t H^{-2}_\kappa}^2 
 + c^{-1} \snorm{h}_{L^2_t H^{-1}_\kappa}^2 \big) \|f\|_{L^\infty_{t,x}} ,
\]
from which \eqref{LSk 3 new} yields
\begin{equation*}
\begin{aligned}
&\bnorm{ \sqrt{R_0}f\psi^3 (R_0u)^3 \sqrt{R_0} }_{L^1_t\tc} \\
&\qquad\lesssim  
\kappa^{-2-\frac32} \snorm{u}_{L^\infty_t H^{-1}_\kappa}
 \big( 
 c^3 \kappa^{-2} \snorm{u}_{\ls_\kappa}^2 
 + (c^{-1}+c^3 \kappa^{-2}) \snorm{u}_{L^2_t H^{-1}_\kappa}^2 \big)  \|f\|_{L^\infty_{t,x}}
 \\ 
&\qquad \lesssim  
 \kappa^{-4-\frac32} \snorm{u}_{L^\infty_t H^{-1}_\kappa}
 \big( c^3
 \snorm{u}_{\ls_\kappa}^2 
 + T (c^{-1}\kappa^2+c^3) \snorm{u}_{L^\infty_t H^{-1}_\kappa}^2 \big) \|f\|_{L^\infty_{t,x}}.
\end{aligned}
\end{equation*}
Choosing
\begin{equation}\label{which-c}
c=T^{\frac14}\kappa^{\frac12}
\end{equation}
to balance the constants, we conclude
\begin{equation}\label{h3 2}
\begin{aligned}
\bnorm{ \sqrt{R_0}f\psi^3 (R_0u)^3 \sqrt{R_0} }_{L^1_t\tc} \lesssim & \ 
\kappa^{-\frac{11}2} (T\kappa^2)^{\frac34} \snorm{u}_{L^\infty_t H^{-1}_\kappa}
\big( \norm{u}_{L^\infty_tH^{-1}_\kappa}^2 + \norm{u}_{\ls_\kappa}^2 \big)  \|f\|_{L^\infty_{t,x}} .
\end{aligned}
\end{equation}
This suffices for the $\ell = 3$ term of~\eqref{h3}.

\bigskip

Next, we turn to the terms with $\ell\geq 4$.  Using \eqref{h3 2}, we estimate
\begin{align*}
&\bnorm{ \sqrt{R_0}f\psi^3 (R_0u)^\ell \sqrt{R_0} }_{L^1_t\tc} \\
&\qquad \leq \bnorm{ \sqrt{R_0} f\psi^3 (R_0u)^3 \sqrt{R_0} }_{L^1_t\tc} \bnorm{\sqrt{R_0}u\sqrt{R_0}}_{L^\infty_{t}\hs}^{\ell - 3} \\
&\qquad \lesssim \kappa^{-\frac{11}2} (T\kappa^2)^{\frac34}\norm{f}_{L^\infty_{t,x}} \norm{u}_{L^\infty_tH^{-1}_\kappa} \big( \norm{u}_{L^\infty_tH^{-1}_\kappa}^2 + \norm{u}_{\ls_\kappa}^2 \big) \big( \kappa^{-\frac{1}{2}} \norm{u}_{L^\infty_tH^{-1}_\kappa} \big)^{\ell - 3} .
\end{align*}
Choosing $C$ larger if necessary, the condition \eqref{k0} implies
\[
\kappa^{-\frac{1}{2}} \norm{u}_{L^\infty_tH^{-1}_\kappa} \leq \tfrac12.
\]
This yields a convergent geometric series that can be summed for $\ell\geq 3$.  Taking the supremum over $\norm{f}_{L^\infty_{t,x}} \leq 1$, the estimate \eqref{h3} follows.
\end{proof}

Using a similar argument, we also obtain estimates in $L^2_tH^1_\kappa$ (instead of $L^1_{t,x}$) that require only one copy of the local smoothing norm:
\begin{lemma}
Given a Schwartz function $\phi(x)$, we have
\begin{equation}
\begin{aligned}
\bnorm{ \phi ( g - \tfrac{1}{2\kappa} - h_1 ) }_{L^2_tH^1_\kappa}
&\leq \sum_{\ell \geq 2} \norm{\phi h_\ell}_{L^2_tH^1_\kappa} \\
& \lesssim \kappa^{-\frac52} (T\kappa^2)^\frac14   \norm{u}_{L^\infty_tH^{-1}_\kappa} \big( \norm{u}_{L^\infty_tH^{-1}_\kappa} + \norm{u}_{\ls_\kappa} \big)
\end{aligned}
\label{h2}
\end{equation}
uniformly for $T\leq\kappa^{-2}$ and $\kappa$ satisfying \eqref{k0}. 
\end{lemma}
\begin{proof} Using a continuous partition of unity argument as in the proof of \eqref{LSk 3 new}, it suffices to prove \eqref{h2} in the special case $\phi = \psi^2$.  From the expansion \eqref{g} for $g$, we have
\[
g - \tfrac{1}{2\kappa} - h_1 = \sum_{\ell \geq 2} h_\ell.
\]
Arguing by duality and taking $f\in L^2_tH^{-1}_\kappa([-T,T]\times\R)$, we find
\[
|\langle  \phi ( g - \tfrac{1}{2\kappa} - h_1 ), f \rangle | \lesssim \sum_{\ell  \geq 2} \| \sqrt{R_0} f\psi^2 (R_0 u)^2 \sqrt{R_0} \|_{L^1_t \tc}.
\]

We begin with the $\ell = 2$ term on the RHS above, and expand
\begin{equation*}
 \sqrt{R_0} f\psi^2 (R_0 u)^2 \sqrt{R_0} 
= \sqrt{R_0} f [\psi^2R_0\tfrac{1}{\psi^2}] \psi u [\psi R_0\tfrac{1}{\psi}] \psi u \sqrt{R_0} .
\end{equation*}
By \eqref{comm}, each operator in square brackets above is bounded $H^{-1}_\kappa\to H^1_\kappa$.  This allows us to write
\begin{equation*}
 \sqrt{R_0} f\psi^2 (R_0 u)^2 \sqrt{R_0} 
=  \big( \sqrt{R_0} f \sqrt{R_0} \big) A_2 \big( \sqrt{R_0} \psi u \sqrt{R_0} \big) A_1 \big( \sqrt{R_0} \psi u \sqrt{R_0} \big) 
\end{equation*}
for operators $A_j$ with $\norm{A_j}_{\op}\lesssim 1$.  We use \eqref{hs-bd} for the $f$ 
contribution and for one copy of $\psi u$, and \eqref{infty-better} for the 
other $\psi u$ contribution:
\[
\begin{aligned}
\bnorm{\sqrt{R_0} f\psi^2 (R_0 u)^2 \sqrt{R_0} }_{\tc} \lesssim & \ 
\kappa^{-\frac32} \|f\|_{H^{-1}_\kappa}
\| \psi u\|_{H^{-1}_\kappa}^\frac32 
\| (\psi u)'\|_{H^{-2}_\kappa}^\frac12 
\\
\lesssim & \ 
\kappa^{-\frac32} \|f\|_{H^{-1}_\kappa}
\| \psi u\|_{H^{-1}_\kappa} \big( c^{-1} \| \psi u\|_{H^{-1}_\kappa} +
 c \| (\psi u)'\|_{H^{-2}_\kappa} \big) .
\end{aligned}
\]
Integrating over time and recalling the bound in
\eqref{LSk 3 new} leads to
\[
\begin{aligned}
&\bnorm{\sqrt{R_0} f\psi^2 (R_0 u)^2 \sqrt{R_0} }_{L^1_t \tc} \\
&\qquad\lesssim \kappa^{-\frac32} \snorm{f}_{L^2_t H^{-1}_\kappa}\| \psi u\|_{L^\infty_t H^{-1}_\kappa}
 \big( \kappa^{-1} c 
  \snorm{u}_{\ls_\kappa} 
 +  (c^{-1} + \kappa^{-1} c)  \snorm{u}_{L^2_t H^{-1}_\kappa} \big).
\end{aligned}
\]
Then, by applying H\"older's inequality in time, we obtain
\[
\begin{aligned}
\bnorm{\sqrt{R_0} f\psi^2 (R_0 u)^2 \sqrt{R_0} }_{L^1_t \tc}
\lesssim \kappa^{-\frac52} \snorm{f}_{L^2_t H^{-1}_\kappa}\| \psi u\|_{L^\infty_t H^{-1}_\kappa}
 \big( c
  \snorm{u}_{\ls_\kappa} 
 +  T^\frac12(\kappa c^{-1} +  c)  \snorm{u}_{L^\infty_t H^{-1}_\kappa} \big).
\end{aligned}
\]
By optimizing the choice of $c$, as in \eqref{which-c}, yields
\[
\begin{aligned}
\bnorm{\sqrt{R_0} f\psi^2 (R_0 u)^2 \sqrt{R_0} }_{L^1_t \tc}
\lesssim \kappa^{-2} T^\frac14 \snorm{f}_{L^2_t H^{-1}_\kappa}\|  u\|_{L^\infty_t H^{-1}_\kappa}
 \big( \norm{u}_{L^\infty_tH^{-1}_\kappa} + \norm{u}_{\ls_\kappa} \big) ,
\end{aligned}
\]
which suffices for the $\ell = 2$ term.

For $\ell\geq 3$, we estimate
\begin{align*}
&\bnorm{ \sqrt{R_0} f\psi^2 (R_0 u)^\ell \sqrt{R_0} }_{L^1_t\tc}
\leq \bnorm{ \sqrt{R_0} f\psi^2 (R_0 u)^2 \sqrt{R_0} }_{L^1_t\tc} \bnorm{ \sqrt{R_0} u \sqrt{R_0} }_{L^\infty_t\hs}^{\ell - 2} \\
&\qquad \lesssim \kappa^{-2} T^\frac14  \norm{f}_{L^2_tH^{-1}_\kappa}  \norm{u}_{L^\infty_tH^{-1}_\kappa}\big( \norm{u}_{L^\infty_tH^{-1}_\kappa} + \norm{u}_{\ls_\kappa} \big) \big( \kappa^{-\frac{1}{2}} \norm{u}_{L^\infty_tH^{-1}_\kappa} \big)^{\ell - 2} .
\end{align*}
Summing over $\ell\geq 2$ (exactly as in the proof of the previous lemma) and taking a supremum over $\norm{f}_{L^2_tH^{-1}_\kappa} \leq 1$ yields the second inequality in \eqref{h2}. 
\end{proof}

Altogether, the previous lemmas provide us with the control we need over the functional $u\mapsto g$.  However, the other key functional in our analysis is $u \mapsto \tfrac{1}{g}$; indeed, this is what appears in the definition \eqref{alpha} of $\rho$.  The following lemma provides us with an estimate for the quadratic and higher order terms of $\tfrac{1}{g}$, using the tools that we have already developed.
\begin{lemma}
Given a Schwartz function $\phi(x)$, we have
\begin{equation}
\bnorm{ \phi \big( \tfrac{1}{g} - 2\kappa + \tfrac{4\kappa^2 h_1}{1+2\kappa h_1} \big) }_{L^2_{t,x}} \lesssim \kappa^{-\frac32} (T\kappa^2)^\frac14
 \norm{u}_{L^\infty_tH^{-1}_\kappa} \big( \norm{u}_{L^\infty_tH^{-1}_\kappa} + \norm{u}_{\ls_\kappa} \big)
\label{1/g est 2}
\end{equation}
uniformly for $T\leq\kappa^{-2}$ and $\kappa$ as in \eqref{k0}.
\end{lemma}
\begin{proof}
Recall that $h_1 = - \kappa^{-1}R_0(2\kappa) u$.  Choosing $C$ larger if necessary, the condition~\eqref{k0} implies
\begin{equation*}
\norm{ 2\kappa h_1 }_{L^\infty_x}
=  \norm{ 2 R_0(2\kappa) u }_{L^\infty_x}
\lesssim \kappa^{-\frac12} \norm{u}_{H^{-1}_\kappa} 
< 1. 
\end{equation*}
Hence,
\begin{equation}
\left\| \dfrac{1}{1+2\kappa h_1} \right\|_{L^\infty_{t,x}} + \left\| \dfrac{2\kappa h_1}{1+2\kappa h_1} \right\|_{L^\infty_{t,x}}
\lesssim 1 
\label{h1 4}
\end{equation}
uniformly for $\kappa$ satisfying \eqref{k0}.

Now, we use the identity
\begin{equation*}
\frac{1}{g} - 2\kappa + \frac{4\kappa^2 h_1}{1+2\kappa h_1}
= -\frac{2\kappa}{g} \left( 1 - \frac{2\kappa h_1}{1+2\kappa h_1} \right) \left( g -\frac{1}{2\kappa} - h_1 \right)  
\end{equation*}
together with the estimates \eqref{1/g est} and \eqref{h2} to bound
\begin{align*}
\bnorm{ \phi \big( \tfrac{1}{g} - 2\kappa + \tfrac{4\kappa^2 h_1}{1+2\kappa h_1} \big) }_{L^2_{t,x}} 
&\lesssim \kappa \bnorm{\tfrac{1}{g}}_{L^\infty_{t,x}} \bnorm{ \phi \big( g -\tfrac{1}{2\kappa} - h_1 \big) }_{L^2_{t,x}} \\
& \lesssim \kappa \bnorm{ \psi \big( g -\tfrac{1}{2\kappa} - h_1 \big) }_{L^2_tH^1_\kappa}
\\
&\lesssim  \kappa^{-1} T^\frac14   \norm{u}_{L^\infty_tH^{-1}_\kappa} \big( \norm{u}_{L^\infty_tH^{-1}_\kappa} + \norm{u}_{\ls_\kappa} \big) .
\qedhere
\end{align*}
\end{proof}

Before proceeding, we will record one more estimate for the product $h_1h_2$, similar in spirit to the preceding analysis.  (This will be useful in estimating \eqref{j3 2}.)
\begin{lemma}
Given a Schwartz function $\phi(x)$, we have
\begin{equation}
\norm{\phi h_1h_2 }_{L^1_{t,x}} \lesssim \kappa^{-\frac{13}2} (T\kappa^2)^\frac34\norm{u}_{L^\infty_t H^{-1}_\kappa} \big( \norm{u}_{L^\infty_tH^{-1}_\kappa}^2 + \norm{u}_{\ls_\kappa}^2 \big)
\label{h1h2}
\end{equation}
uniformly for $T\leq\kappa^{-2}$ and $\kappa\geq 1$.
\end{lemma}
\begin{proof}
Using a continuous partition of unity argument as in the proof of \eqref{LSk 3 new}, it suffices to prove \eqref{h1h2} in the special case where $\phi = \psi^3$.  

We argue by duality: for $f\in L^\infty([-T,T]\times\R)$, we have
\begin{equation*}
\int_{-T}^T \int_{-\infty}^\infty f\psi^3 h_1h_2 \dx\dt
= \int_{-T}^T \tr\big\{ \sqrt{R_0} f\psi^3 h_1 (R_0u)^2 \sqrt{R_0} \big\} \dt .
\end{equation*}
By \eqref{comm}, we may write
\begin{equation*}
\sqrt{R_0} f\psi^3 h_1 (R_0u)^2 \sqrt{R_0} 
= \sqrt{R_0} f \psi h_1 \sqrt{R_0} A_2 \sqrt{R_0} \psi u \sqrt{R_0} A_1 \sqrt{R_0} \psi u \sqrt{R_0} 
\end{equation*}
for operators $A_j$ with $\norm{A_j}_{\op}\lesssim 1$.

Now we use the bounds \eqref{hs-bd} and \eqref{infty-better}
to  estimate
\[
\bnorm{\sqrt{R_0} f\psi^3 h_1 (R_0u)^2 \sqrt{R_0}}_{\tc} \lesssim 
\kappa^{-\frac32} \| f \psi h_1\|_{H^{-1}_\kappa} \| \psi u\|_{H^{-1}_\kappa}^\frac32 \| (\psi u)'\|_{H^{-2}_\kappa}^\frac12 .
\]
As $h_1 = -\kappa^{-1} R_0(2\kappa)u$,  the first factor on the right is estimated as 
\[
\| f \psi h_1\|_{H^{-1}_\kappa} 
\lesssim  \kappa^{-1} \| f \|_{L^\infty} \|\psi h_1\|_{L^2} 
\lesssim  \kappa^{-3} \| f \|_{L^\infty} \|\psi R_0 u\|_{H^1_\kappa}
\lesssim  \kappa^{-3} \| f \|_{L^\infty} \|\psi  u\|_{H^{-1}_\kappa} .
\]
Hence we obtain the fixed-time bound
\[
\begin{aligned}
\bnorm{\sqrt{R_0} f\psi^3 h_1 (R_0u)^2 \sqrt{R_0}}_{\tc} \lesssim & \ 
\kappa^{-\frac92} \| f \|_{L^\infty} \| \psi u\|_{H^{-1}_\kappa}^\frac52 \| (\psi u)'\|_{H^{-2}_\kappa}^\frac12
\\
\lesssim & \ 
\kappa^{-\frac92} \| f \|_{L^\infty} \| \psi u\|_{H^{-1}_\kappa} \big( c^3 \| (\psi u)'\|_{H^{-2}_\kappa}^2 + c^{-1} \| \psi u\|_{H^{-1}_\kappa}^2 \big) .
\end{aligned}
\]
Integrating in time and using \eqref{LSk 3 new}, this yields
\[
\begin{aligned}
&\bnorm{\sqrt{R_0} f\psi^3 h_1 (R_0u)^2 \sqrt{R_0}}_{L^1_t \tc} \\
&\qquad\lesssim 
\kappa^{-\frac{13}2} \| f \|_{L^\infty_{t,x}} \| \psi u\|_{L^\infty_tH^{-1}_\kappa} \big( c^3 \| u \|_{\ls_\kappa}^2 + T ( c^{-1} \kappa^{2} + c^3) \| \psi u\|_{L^\infty_t H^{-1}_\kappa}^2 \big) .
\end{aligned}
\]
Balancing $c$ as in \eqref{which-c},
we arrive at
\[
\begin{aligned}
\bnorm{\sqrt{R_0} f\psi^3 h_1 (R_0u)^2 \sqrt{R_0}}_{L^1_t \tc} \lesssim & \ 
\kappa^{-\frac{13}2} (T\kappa^2)^\frac34 \| f \|_{L^\infty_{t,x}} \| \psi u\|_{L^\infty_tH^{-1}_\kappa} \big( \norm{u}_{L^\infty_tH^{-1}_\kappa}^2 + \norm{u}_{\ls_\kappa}^2 \big) 
\end{aligned}
\]
as required.
\end{proof}

We conclude this section with two estimates for the functional derivative of $\rho$, which appears on RHS\eqref{alpha micro 2}.  First, we present an estimate that does not make use of local smoothing:
\begin{lemma}
Given $T>0$, we have
\begin{equation}
\bnorm{  d\rho|_u(f) }_{L^1_{t,x}} \lesssim \kappa^{-1} \norm{u}_{L^\infty_tH^{-1}_\kappa} \norm{f}_{L^1_tH^{-1}_\kappa}
\label{dg H-1}
\end{equation}
uniformly for $\kappa$ satisfying \eqref{k0}. 
\end{lemma}
\begin{proof}
Arguing by duality, let $h\in L^\infty([-T,T]\times\R)$.  Expanding the series \eqref{Ru-exp} in the definition \eqref{dg} of $dg$, we obtain
\begin{equation*}
\int_{-T}^T \int h \tfrac{1}{g^2} dg|_u(f) \dx\dt 
= \sum_{\ell,m\geq 0} \int_{-T}^T (-1)^{\ell+m+1} \tr\big\{ \tfrac{h}{g^2} (R_0 u)^\ell R_0 f (R_0 u)^m R_0 \big\} \dt .
\end{equation*}
This is the main term in the formula \eqref{drho} for $d\rho$.

When $\ell = m =0$, we have
\begin{align*}
-\tr\big\{ \tfrac{h}{g^2} R_0 f R_0 \big\}
&= -\int \tfrac{h}{g^2}\, \kappa^{-1}R_0(2\kappa)f \dx \\
&= -\int h\, 4\kappa R_0(2\kappa)f \dx - \int h\big( \tfrac{1}{g^2} - 4\kappa^2 \big) \kappa^{-1} R_0(2\kappa)f \dx .
\end{align*}
The first term on the RHS above is canceled out by the contribution of $2\kappa R_0(2\kappa) f$ in \eqref{drho}.  It then remains to estimate the second term, for which we use \eqref{1/g est}:
\begin{align*}
&\int_{-T}^T \int \big| h\big( \tfrac{1}{g^2} - 4\kappa^2 \big) \kappa^{-1} R_0(2\kappa)f \big| \dx \dt \\
&\qquad \lesssim \norm{h}_{L^\infty_{t,x}} \snorm{ \tfrac{1}{g} - 2\kappa}_{L^\infty_tL^2_x} \big( \snorm{\tfrac{1}{g}}_{L^\infty_{t,x}} + 2\kappa \big) \kappa^{-1} \norm{R_0(2\kappa) f}_{L^1_tL^2_x} \\
&\qquad \lesssim \kappa^{-1}\norm{h}_{L^\infty_{t,x}} \norm{u}_{L^\infty_tH^{-1}_\kappa} \norm{f}_{L^1_tH^{-1}_\kappa} .
\end{align*} 

Next, we turn to the terms with $\ell + m \geq 1$.  In this case, there are at least two operators we can put in Hilbert--Schmidt norm:
\begin{align*}
&\sum_{ \ell+m\geq 1 } \bnorm{ \tfrac{h}{g^2} (R_0u)^\ell R_0 f (R_0u)^m R_0 }_{L^1_t\tc} \\
&\qquad \leq \sum_{ \ell+m\geq 1 } \bnorm{ \sqrt{R_0}\tfrac{h}{g^2} \sqrt{R_0} }_{L^\infty_t\op} \bnorm{ \sqrt{R_0}f\sqrt{R_0} }_{L^1_t\hs} \bnorm{\sqrt{R_0}u\sqrt{R_0}}_{L^\infty_t\hs}^{\ell+m} \\
&\qquad\lesssim \kappa^{-\frac52} \sum_{ \ell+m\geq 1 } \bnorm{\tfrac{h}{g^2}}_{L^\infty_{t,x}} \norm{f}_{L^1_tH^{-1}_\kappa} \big( \kappa^{-\frac{1}{2}} \norm{u}_{L^\infty_tH^{-1}_\kappa} \big)^{\ell+m} \\
&\qquad \lesssim \kappa^{-1} \norm{u}_{L^\infty_tH^{-1}_\kappa} \norm{h}_{L^\infty_{t,x}} \norm{f}_{L^1_tH^{-1}_\kappa} .
\end{align*}

Altogether, we have
\begin{equation*}
\int_{-T}^T \int \big| h\,d\rho|_u(f) \big|\dx\dt 
\lesssim \kappa^{-1} \norm{u}_{L^\infty_tH^{-1}_\kappa}  \norm{h}_{L^\infty_{t,x}} \norm{f}_{L^1_tH^{-1}_\kappa} .
\end{equation*}
Taking a supremum over $\norm{h}_{L^\infty_{t,x}}\leq 1$, we obtain \eqref{dg H-1}.
\end{proof}

We can also handle one derivative inside $d\rho$, provided that the input is localized in space and we may estimate one copy of $u$ in local smoothing norm:
\begin{lemma}
We have
\begin{equation}
 \bigg| \int_{-T}^T \int \phi_{x_0}\, d\rho|_u (\psi^2f)'\dx\dt \bigg| 
\lesssim \kappa^{-1} \norm{\psi f}_{L^2_tH^{-1}_\kappa} \big( \norm{u}_{L^\infty_tH^{-1}_\kappa} + \norm{u}_{\ls_\kappa} \big) 
\label{dg LSk}
\end{equation}
uniformly for $x_0\in\R$, $T\leq\kappa^{-2}$, and $\kappa$ satisfying \eqref{k0}, where $\phi$ is defined by
\begin{equation}
\phi_{x_0}(x) = \phi(x-x_0) = 6\operatorname{tanh}(\tfrac{x-x_0}{6}) \quad\tx{so that}\quad \phi'_{x_0} = \psi^2_{x_0} .
\label{switch fnc}
\end{equation}
\end{lemma}
\begin{proof}
In the following, we will set $\phi = \phi_{x_0}$ for simplicity.  First, we use \eqref{drho} to write
\begin{align}
2\int \phi\, d\rho|_u (\psi^2f)'\dx 
&= \int \phi \Big\{ \tfrac{1}{g^2} dg|_u(\psi^2f)' + 4\kappa R_0(2\kappa) (\psi^2f)' \Big\} \dx 
\nonumber \\
& = \int \tfrac{\phi}{g^2} \Big\{  dg|_u(\psi^2f)' + \kappa^{-1} R_0(2\kappa) (\psi^2f)' \Big\} \dx 
\label{dg LSk 1}\\
&\quad- \kappa^{-1} \int \phi \big( \tfrac{1}{g} + 2\kappa \big) \big( \tfrac{1}{g} - 2\kappa + \tfrac{4\kappa^2h_1}{1+2\kappa h_1} \big)  R_0(2\kappa) (\psi^2f)' \dx 
\label{dg LSk 2}\\
&\quad+ \kappa^{-1} \int \phi \big( \tfrac{1}{g} + 2\kappa \big) \tfrac{4\kappa^2h_1}{1+2\kappa h_1} R_0(2\kappa) (\psi^2f)' \dx .
\label{dg LSk 3}
\end{align}
We will estimate the contributions of \eqref{dg LSk 1}--\eqref{dg LSk 3} individually.

Let us start with \eqref{dg LSk 2}.  By \eqref{1/g est}, \eqref{1/g est 2}, and \eqref{comm}, we have
\begin{align*}
&\kappa^{-1} \int_{-T}^T \bigg| \int \phi \big( \tfrac{1}{g} + 2\kappa \big) \big( \tfrac{1}{g} - 2\kappa + \tfrac{4\kappa^2h_1}{1+2\kappa h_1} \big)  R_0(2\kappa) (\psi^2f)' \dx \bigg| \dt \\
&\qquad \lesssim \kappa^{-1} \norm{\phi}_{L^\infty} \bnorm{ \tfrac{1}{g} + 2\kappa }_{L^\infty_{t,x}} \bnorm{ \psi \big( \tfrac{1}{g} - 2\kappa + \tfrac{4\kappa^2 h_1}{1+2\kappa h_1} \big) }_{L^2_{t,x}} \bnorm{ \tfrac{1}{\psi} R_0\partial \psi }_{H^{-1}_\kappa\to L^2} \norm{\psi f}_{L^2_tH^{-1}_\kappa} \\
&\qquad\lesssim \kappa^{-\frac32} (T\kappa^2)^\frac14  \norm{u}_{L^\infty_tH^{-1}_\kappa}
 \big( \norm{u}_{L^\infty_tH^{-1}_\kappa} + \norm{u}_{\ls_\kappa} \big)\norm{\psi f}_{L^2_tH^{-1}_\kappa} .
\end{align*} 

Next, we turn to \eqref{dg LSk 3}.  Integrating by parts and using \eqref{comm}, \eqref{1/g est}, \eqref{h1 2}, \eqref{h1 4}, \eqref{LSk g}, and \eqref{h1 1}, we have
\begin{align*}
&\kappa^{-1} \int_{-T}^T \bigg| \int \phi \big( \tfrac{1}{g} + 2\kappa \big) \tfrac{4\kappa^2h_1}{1+2\kappa h_1} R_0(2\kappa) (\psi^2f)' \dx \bigg| \dt \\
&\qquad \lesssim \kappa^{-1} \norm{ \psi f }_{L^2_tH^{-1}_\kappa} \bnorm{ \psi R_0 \tfrac{1}{\psi} }_{L^2 \to H^1_\kappa} \bnorm{ \psi \big[ \phi \big( \tfrac{1}{g} + 2\kappa \big) \tfrac{4\kappa^2h_1}{1+2\kappa h_1} \big]' }_{L^2_{t,x}} \\
&\qquad \lesssim \kappa^{-2} \norm{ \psi f }_{L^2_tH^{-1}_\kappa} \Big\{ \bnorm{ \tfrac{1}{g} + 2\kappa }_{L^\infty_{t,x}} \norm{ \psi h'_1 }_{L^2_{t,x}} \bnorm{ \tfrac{4\kappa^2}{1+2\kappa h_1} }_{L^\infty_{t,x}} \big( 1 + \bnorm{ \tfrac{2\kappa h_1}{1+2\kappa h_1} }_{L^\infty_{t,x}} \big) \\
&\qquad\qquad + \norm{ \psi g'}_{L^2_{t,x}} \bnorm{\tfrac{1}{g}}_{L^\infty_{t,x}}^2 \bnorm{ \tfrac{4\kappa^2h_1}{1+2\kappa h_1} }_{L^\infty_{t,x}}  
+ \bnorm{ \tfrac{1}{g} + 2\kappa }_{L^\infty_{t,x}} \bnorm{ \tfrac{4\kappa^2}{1+2\kappa h_1} }_{L^\infty_{t,x}} \norm{\psi h_1}_{L^2_{t,x}} \Big\}  \\
&\qquad\lesssim \kappa^{-1} \norm{\psi f}_{L^2_tH^{-1}_\kappa}  \big( \norm{u}_{L^\infty_tH^{-1}_\kappa} + \norm{u}_{\ls_\kappa} \big) .
\end{align*} 

It remains to estimate \eqref{dg LSk 1}.  We expand
\begin{align*}
&\int \tfrac{\phi}{g^2} \Big\{  dg|_u(\psi^2f)' + \kappa^{-1} R_0(2\kappa) (\psi^2f)' \Big\} \dx \\
&\qquad = \sum_{\ell + m \geq 1}(-1)^{\ell+m+1} \tr\big\{ \tfrac{\phi}{g^2} (R_0u)^\ell R_0  \psi^2 [\partial, \psi^2 f] (R_0u)^m R_0 \big\} \\
&\qquad = \sum_{\ell + m \geq 1} (-1)^{\ell+m} \tr\big\{ \psi^2f \big[ \partial, (R_0u)^\ell R_0 \tfrac{\phi}{g^2} (R_0u)^m R_0 \big] \big\} .
\end{align*}
In the last equality, we cycled the trace (i.e.\ $\tr(AB) = \tr(BA)$).  Now, we will distribute the derivative $\partial$ using the product rule $[\partial, AB] = [\partial, A]B + A[\partial,B]$.

First, let us consider the terms with $\ell+m=1$.  When the derivative lands on $u$, we use \eqref{comm} and \eqref{1/g est} to estimate:
\begin{align*}
&\int_{-T}^T \big| \tr\big\{ \psi^2 f R_0u'R_0 \tfrac{\phi}{g^2} R_0 \big\} \big| + \big| \tr\big\{ \psi^2 f R_0 \tfrac{\phi}{g^2} R_0u'R_0 \big\} \big| \dt \\
&\qquad\lesssim \bnorm{ \sqrt{R_0} \psi f \sqrt{R_0} }_{L^2_t\hs} \bnorm{ \sqrt{R_0} \psi u' \sqrt{R_0} }_{L^2_t\hs} \bnorm{ \sqrt{R_0}  \tfrac{\phi}{g^2} \sqrt{R_0} }_{L^\infty_t\op} \\
&\qquad\lesssim \kappa^{-1} \norm{\psi f}_{L^2_tH^{-1}_\kappa} \norm{u}_{\ls_\kappa} .
\end{align*}
When the derivative lands on $\tfrac{\phi}{g^2}$, we use \eqref{HS}, \eqref{1/g est}, and \eqref{k0} to estimate
\begin{equation}
\begin{aligned}
\bnorm{ \sqrt{R_0} (\tfrac{\phi}{g^2})' \sqrt{R_0} }_{L^2_t\op}
&\lesssim \kappa^{-\frac32} T^{\frac12} \bnorm{ (\tfrac{\phi}{g^2})' }_{L^\infty_tL^2_x} 
\lesssim \kappa^{-\frac{3}{2}} T^{\frac12} \big( \bnorm{ \tfrac{1}{g} }_{L^\infty_{t,x}}^2 + \bnorm{ \tfrac{1}{g} }_{L^\infty_{t,x}}^3 \norm{ g'}_{L^2_{t,x}} \big) \\
&\lesssim \kappa^{\frac{1}{2}} 
 T^{\frac12} \big( 1+ \norm{u}_{L^\infty_tH^{-1}_\kappa} \big)
\lesssim 1 .
\end{aligned}
\label{1/g est op}
\end{equation}
This yields
\begin{align*}
&\int_{-T}^T \big| \tr\big\{ \psi^2 f R_0uR_0 \big(\tfrac{\phi}{g^2}\big)' R_0 \big\} \big| + \big| \tr\big\{ \psi^2 f R_0 \big(\tfrac{\phi}{g^2}\big)' R_0uR_0 \big\} \big| \dt \\
&\qquad\lesssim \bnorm{ \sqrt{R_0} \psi^2 f \sqrt{R_0} }_{L^2_t\hs} \bnorm{ \sqrt{R_0} u \sqrt{R_0} }_{L^\infty_t\hs} \bnorm{ \sqrt{R_0}(\tfrac{\phi}{g^2})' \sqrt{R_0} }_{L^2_t\op} \\
&\qquad\lesssim \kappa^{-1} \norm{\psi f}_{L^2_tH^{-1}_\kappa} \norm{u}_{L^\infty_tH^{-1}_\kappa} .
\end{align*}

Finally, we turn to the terms with $\ell + m\geq 2$.  When the derivative lands on $u$ we use \eqref{LSk H1k 2}, and when the derivative lands on $\tfrac{\phi}{g^2}$ we use \eqref{1/g est op}.    
As $\ell+m\geq 2$, there are always at least two operators that are put in Hilbert--Schmidt norm:
\begin{align*}
&\int_{-T}^T \big| \tr\big\{ \psi^2f \big[ \partial, (R_0u)^\ell R_0 \tfrac{\phi}{g^2} (R_0u)^m R_0 \big] \big\} \big| \dt \\
&\qquad\lesssim \bnorm{\sqrt{R_0} \psi^2 f \sqrt{R_0} }_{L^2_t\hs} \Big\{ \bnorm{ \sqrt{R_0} \big( \tfrac{\phi}{g^2} \big)' \sqrt{R_0} }_{L^2_t\op} \bnorm{ \sqrt{R_0} u \sqrt{R_0} }_{L^\infty_t\hs}^{\ell+m} \\
&\qquad\qquad+ (\ell + m) \bnorm{ \sqrt{R_0} u' \sqrt{R_0} }_{L^2_t\op} \bnorm{ \sqrt{R_0} \tfrac{\phi}{g^2} \sqrt{R_0} }_{L^\infty_t\op} \bnorm{ \sqrt{R_0} u \sqrt{R_0} }_{L^\infty_t\hs}^{\ell+m-1}   \Big\} \\
&\qquad\lesssim \kappa^{-1} \norm{\psi f}_{L^2_tH^{-1}_\kappa} \big( \norm{u}_{L^\infty_tH^{-1}_\kappa} + \norm{u}_{\ls_\kappa} \big) (\ell+m+1) \big( \kappa^{-\frac{1}{2}} \norm{u}_{L^\infty_tH^{-1}_\kappa} \big)^{\ell+m-1} .
\end{align*}
Summing over $\ell + m \geq 2$ then yields
\begin{align*}
&\sum_{\ell + m \geq 2} \int_{-T}^T \big| \tr\big\{ \psi^2f \big[ \partial, (R_0u)^\ell R_0 \tfrac{\phi}{g^2} (R_0u)^m R_0 \big] \big\} \big| \dt  \lesssim \kappa^{-1} \norm{\psi f}_{L^2_t H^{-1}_\kappa} \big( \norm{u}_{L^\infty_tH^{-1}_\kappa} + \norm{u}_{\ls_\kappa} \big) .
\end{align*}
This was the final term that we needed to estimate, and thus concludes the proof of the lemma.
\end{proof}

\section{The local smoothing estimate}
\label{s:ls est}

In this section, we will prove our main local smoothing estimate for solutions to \eqref{gKdV}:
\begin{theorem}
\label{t:local smoothing}
Let $\kappa \geq 1$, and $u$ be a solution to 
\eqref{gKdV} which satisfies the bound \eqref{k0} uniformly in a time interval $[0,T]$ with $T \leq \kappa^{-2}$. 
For some $\eps > 0$, suppose that the coefficients $a_1,a_2,a_3,a_4$ obey the following bounds in $[0,T]$: 
\begin{gather}
\int \norm{\psi_z a_1}_{L^\infty_tH^1_x} \dz \leq \eps 
\quad\tx{or}\quad
\int \norm{\psi_z a_1}_{L^\infty_tW^{1,\infty}_x} \dz \leq \eps ,
\label{a}\\
\int \norm{\psi_z a_2}_{L^\infty_{t,x}} \dz \leq \eps ,
\label{b}\\
\int \norm{\psi_z a_3}_{L^2_tH^1_x} \dz \leq \eps 
\quad\tx{or}\quad
\int \norm{\psi_z a_3}_{L^2_tW^{1,\infty}_x} \dz \leq \eps ,
\label{c}\\
\norm{a_4}_{ L^\infty_tH^1_x} < \infty
\quad\tx{or}\quad
\norm{a_4}_{ L^\infty_tW^{1,\infty}_x} < \infty .
\label{d}
\end{gather}
Then the solution $u(t)$ to \eqref{gKdV} satisfies the local energy estimate
\begin{equation}
\norm{u}_{\ls_\kappa}^2 \leq 2\norm{u}_{C_tH^{-1}_\kappa}^2 + C(\eps+ (T\kappa^2)^{\frac14} + \kappa^{-2})  \big( \norm{u}_{L^\infty_tH^{-1}_\kappa}^2 + \norm{u}_{\ls_\kappa}^2 \big)  .
\label{LSk}
\end{equation}

\end{theorem}
For the definition of $\psi_{x_0}$ and $\ls_\kappa$, see~\eqref{LSk def}.
The rest of the section is devoted to the proof of this theorem. 
\begin{proof}[Proof of Theorem \ref{t:local smoothing}]
Recall that solutions $u(t)$ to \eqref{gKdV} obey the approximate conservation law \eqref{alpha micro 2}.  In order to prove our local smoothing estimate, we multiply \eqref{alpha micro 2} by the smooth step function $\phi_{x_0}$ defined in \eqref{switch fnc} and integrate in space and time:
\begin{align*}
\int_{-T}^T \int \psi^2_{x_0} j(t,x)\dx\dt = & \ \int \phi_{x_0} [\rho(T,x) - \rho(-T,x)]\dx   \\
& - \int_{-T}^T \int \phi_{x_0}\, d\rho|_u\big[ (a_1u')' + a_2u^2 + a_3u' + a_4u \big] \dx\dt .
\end{align*}
Ultimately, we seek to identify the left-hand side (respectively, the first term on the right) above
with the left-hand side (respectively, the first term on the right) of \eqref{LSk}, modulo acceptable errors.

We will start with the easiest term, which is the first term on the right. Recall that $\rho(t,x)\geq 0$, and so $\norm{\rho}_{L^1} = \alpha(\kappa,u)$.  Combining this with \eqref{alpha est}, leads to the bound
\begin{equation}\label{rhs1}
\kappa \bigg| \int \phi_{x_0} [\rho(T,x) - \rho(-T,x)]\dx \bigg| \leq 2 \norm{u}_{C_tH^{-1}_\kappa}^2 .
\end{equation}

The rest of the proof is organized as follows.  For the left-hand side, we will show that 
\begin{equation}\label{lhs}
\kappa \int_{-T}^T \int \psi^2_{x_0} j(t,x) =
 -\tfrac{3}{2} \bnorm{ \psi_{x_0} u' }_{L^2_tH^{-1}_\kappa}^2 \dx \dt + \Err_0,
\end{equation}
where $\Err_0$ satisfies
\[
|\Err_0| \lesssim  ({(T\kappa^2})^{\frac14}+\kappa^{-2})  \big( \norm{u}_{L^\infty_tH^{-1}_\kappa}^2 + \norm{u}_{\ls_\kappa}^2 \big).
\]
The last task will then be to estimate the
second term on the right, and show that
\begin{equation}\label{rhs2}
 \kappa \int_{-T}^T \int \phi_{x_0} \, d\rho|_u \big[ (a_1u')' + a_2u^2 + a_3u' + a_4u \big] \dx\dt = \Err_1  , 
\end{equation}
where
\[
|\Err_1| \lesssim (\eps+\kappa^{-2})  \big( \norm{u}_{L^\infty_tH^{-1}_\kappa}^2 + \norm{u}_{\ls_\kappa}^2 \big)  .
\]
The desired estimate \eqref{LSk} is then obtained 
by combining \eqref{rhs1}, \eqref{lhs}, and \eqref{rhs2} and taking the supremum over $x_0 \in \R$.  

\bigskip

The remainder of this section is dedicated to proving \eqref{lhs} and \eqref{rhs2}, and we begin with \eqref{lhs}.  To leading order, $j$ is quadratic in $u$.  Specifically, if we insert the series \eqref{g} for $g$, then the quadratic terms of $j$ are:
\begin{equation*}
j_2 := 8\kappa^4 h_2 - h_1 (16\kappa^5h_1 + 4\kappa^2u) - 6\kappa R_0(2\kappa)[u^2] .
\end{equation*}
It is then natural to expect the leading contribution to come from $j_2$, while the cubic and higher contributions coming from $j-j_2$ to be perturbative. We will realize these goals in the next two lemmas. In this first lemma, we examine the contribution of $j_2$ and show how it produces the local smoothing norm of $u$:
\begin{lemma}
Assume that $\kappa \geq 1$ and $u$ satisfies \eqref{k0} uniformly in $[0,T]$. Then we have
\begin{align*}
\bnorm{ \psi_{x_0} u' }_{L^2_tH^{-1}_\kappa}^2
= & - \tfrac{2}{3} \kappa \int_{-T}^T\int_{-\infty}^\infty \psi_{x_0}^2 j_2 \dx\dt 
+ O\Big(\kappa^{-2} \big( \norm{u}_{L^\infty_t{H^{-1}_\kappa}}^2 + \norm{u}_{\ls_\kappa}^2 \big) \Big)
\end{align*}
uniformly for $x_0\in\R$.
\end{lemma}
\begin{proof}
Evidently, we need $\kappa j_2$ to be $O(1)$, and so there must be some cancellation for the terms that are $O(\kappa)$ and higher.  In order to exhibit this cancellation, we use the identities (cf.~\cite{Bringmann2021}*{Lem.~2.5})
\begin{gather}
16\kappa^5 h_2 = 3u^2 - 3\kappa^2 (h''_1)^2 - 20\kappa^4 \big[ (h'_1)^2 - (h_1^2)'' \big] + 4\kappa^4 \partial^2 R_0(2\kappa) \big[ (h'_1)^2 + 2(h_1^2)'' \big] ,
\nonumber \\
16\kappa^5 h_1 + 4\kappa^2 u = 4\kappa^3h''_1
\label{h1}
\end{gather}
to write
\begin{align*}
-2\kappa j_2 
= 3\kappa^2 (h''_1)^2 + 12\kappa^4 (h'_1)^2 - 3\big[ u^2 - 4\kappa^2 R_0(2\kappa)u^2 \big] - 4\kappa^4 \partial^2 R_0(2\kappa) \big[ (h'_1)^2 + 2(h_1^2)'' \big] .
\end{align*}

We multiply this by $\psi^2_{x_0}$ and integrate in space and time.  Working from left to right, we claim that
\begin{gather}
\begin{aligned}
\int_{-T}^T \langle \psi^2, 3\kappa^2 (h''_1)^2 \rangle \dt
={}& 3 \int_{-T}^T \langle (\psi u')', R_0^2 (\psi u')' \rangle \dt + O\big( \kappa^{-2}  \norm{u}_{\ls_\kappa}^2 \big)  , 
\end{aligned}
\label{j2 1} \\
\begin{aligned}
\int_{-T}^T \langle \psi^2, 12\kappa^4 (h'_1)^2 \rangle \dt
={}& 3 \int_{-T}^T \langle \psi u', R_0 \psi u' \rangle - \langle (\psi u')', R_0^2 (\psi u')' \rangle 
 \dt \\
&+ O\big( \kappa^{-2}  \norm{u}_{\ls_\kappa}^2 \big) , 
\end{aligned}
\label{j2 2}\\
\int_{-T}^T \langle \psi^2, (\tx{rest of }{-2\kappa j_2}) \rangle \dt
= O\Big( \kappa^{-2} \big( \norm{u}_{L^\infty_tH^{-1}_\kappa}^2 + \norm{u}_{\ls_\kappa}^2 \big) \Big)
\label{j2 3}
\end{gather}
uniformly for $x_0\in\R$, where $\psi = \psi_{x_0}$ and $R_0 = R_0(2\kappa)$.  Adding these together, this would yield
\begin{align*}
\int_{-T}^T -2\kappa \langle \psi^2, j_2 \rangle \dt
&= 3\int_{-T}^T \langle \psi u', R_0 \psi u' \rangle \dt + O\Big( \kappa^{-2} \big( \norm{u}_{L^\infty_tH^{-1}_\kappa}^2 + \norm{u}_{\ls_\kappa}^2 \big) \Big) .
\end{align*}
The first term on the RHS is exactly $\snorm{\psi u'}_{L^2_tH^{-1}_\kappa}^2$, and so this would finish the proof.

Let us start with ~\eqref{j2 1}.  As $h_1 = -\kappa^{-1} R_0(2\kappa)u$, we have
\begin{equation*}
\kappa^2 \langle \psi^2, (h''_1)^2\rangle
= \langle \psi R_0 u'', \psi R_0u'' \rangle .
\end{equation*}
Comparing this with \eqref{j2 1}, we see that each $\psi R_0 \partial$ above should be replaced by $R_0 \partial \psi$; this is exactly the purpose of our `double commutator' estimates.
By \eqref{double comm1}, we have
\begin{align*}
&\int_{-T}^T \big| \langle \psi R_0 u'', \psi R_0u'' \rangle - \langle R_0(\psi u')', R_0(\psi u')' \rangle \big| \dt \\
&\qquad = \int_{-T}^T \big| \langle \psi u' , \tfrac{1}{\psi} \big( R_0\partial \psi^2 R_0 \partial- \psi R_0^2 \partial^2 \psi \big) \tfrac{1}{\psi} \cdot \psi u' \rangle \big| \dt \\
&\qquad \lesssim  \|\psi u'\|_{L^2_tH^{-2}_\kappa}^2
\\
&\qquad \lesssim   \kappa^{-2} \norm{u}_{\ls_\kappa}^2 .
\end{align*}
This proves \eqref{j2 1}.
\medskip

The proof of \eqref{j2 2} proceeds in a similar way.  We write
\begin{equation*}
\langle \psi^2, 12\kappa^4 (h'_1)^2 \rangle
= 12\kappa^2 \langle \psi R_0 u', \psi R_0 u' \rangle .
\end{equation*}
First, we use \eqref{double comm}  to replace $\psi R_0$ by $R_0\psi$ above:
\begin{align*}
&12\kappa^2 \int_{-T}^T \big| \langle \psi R_0 u', \psi R_0u' \rangle - \langle R_0\psi u', R_0\psi u' \rangle \big| \dt \\
&\qquad = 12\kappa^2 \int_{-T}^T \big| \langle \psi u' , \tfrac{1}{\psi} \big( R_0\psi^2 R_0 - \psi R_0^2 \psi \big) \tfrac{1}{\psi} \cdot \psi u' \rangle \big| \dt 
\\
&\qquad \lesssim 
\norm{\psi u'}_{L^2_t H^{-2}_\kappa}^2
\\
&\qquad \lesssim    \kappa^{-2} \norm{u}_{\ls_\kappa}^2.
\end{align*}
Next, we use the identity $4\kappa^2 R_0 = 1 + \partial^2 R_0 $ to write
\begin{equation*}
12\kappa^2 \langle R_0\psi u', R_0\psi u' \rangle
= 3 \langle \psi u', R_0 \psi u'\rangle - 3 \langle (\psi u')' , R_0^2 (\psi u')' \rangle ,
\end{equation*}
and then \eqref{j2 2} follows.

Lastly, we turn to \eqref{j2 3}.  There are two terms remaining in $-2\kappa j_2$ which we want to show make a negligible contribution.  For the first term, we use the identity $1 - 4\kappa^2 R_0 = - \partial^2 R_0 $ to write
\begin{equation*}
\langle \psi^2 , 3\big[ u^2 - 4\kappa^2 R_0 u^2 \big] \rangle
= 3 \langle \psi^2 , R_0 (u^2)'' \rangle
= 3 \langle \tfrac{1}{\psi^2} R_0 (\psi^2)'' , \psi^2 u^2 \rangle .
\end{equation*}
Therefore, by \eqref{comm} and \eqref{LSk 4 new}--\eqref{LSk 2 new}, for $T \leq \kappa^{-2}$ we have
\begin{align*}
\int_{-T}^T \big| \langle \psi^2 , 3\big[ u^2 - 4\kappa^2 R_0 u^2 \big] \rangle \big| \dt
&\leq 3 \bnorm{\tfrac{1}{\psi^2} R_0 (\psi^2)'' }_{L^\infty_x} \norm{ \psi u}_{L^2_{t,x}}^2 \\
&\lesssim \kappa^{-2} \bnorm{ (-\partial^2 + 4\kappa^2) (\psi u) }^2_{L^2_tH^{-2}_\kappa} \\
&\lesssim \kappa^{-2} \big( \norm{u}_{L^\infty_t{H^{-1}_\kappa}}^2 + \norm{u}_{\ls_\kappa}^2 \big) .
\end{align*}

For the last term of $-2\kappa j_2$, we write
\begin{equation*}
\langle \psi^2 , 4\kappa^4 \partial^2 R_0 [ (h'_1)^2 + 2 (h_1^2 )'' \big] \rangle
= 4\kappa^2 \langle R_0 (\psi^2)'', (R_0 u')^2 \rangle + 4\kappa^2 \langle R_0 (\psi^2)^{(4)} , (R_0 u)^2 \rangle ,
\end{equation*}
in order to estimate
\begin{align*}
&\int_{-T}^T \big| \langle \psi^2 , 4\kappa^4 \partial^2 R_0 [ (h'_1)^2 + 2 (h_1^2 )'' \big] \rangle \big| \dt \\
&\qquad \lesssim \kappa^2 \Big\{ \bnorm{\tfrac{1}{\psi^2} R_0 (\psi^2)'' }_{L^\infty_x} \bnorm{ \psi R_0 \tfrac{1}{\psi} \cdot \psi u'}_{L^2_{t,x}}^2 + \bnorm{\tfrac{1}{\psi^2} R_0 (\psi^2)^{(4)} }_{L^\infty_x} \bnorm{ \psi R_0 \tfrac{1}{\psi} \cdot \psi u }_{L^2_{t,x}}^2 \Big\} \\
&\qquad \lesssim \kappa^2 \big\{ \kappa^{-2} \norm{ \psi u'}_{L^2_t H^{-2}_\kappa}^2 +  \kappa^{-2} \norm{ \psi u }_{L^2_t H^{-2}_\kappa}^2 \big\}
\lesssim \kappa^{-2} \big( \norm{u}_{L^\infty_t{H^{-1}_\kappa}}^2 + \norm{u}_{\ls_\kappa}^2 \big) .
\end{align*}
Altogether, this proves \eqref{j2 3}.
\end{proof}

In the next lemma, we show that the cubic and higher order terms of $j$ make a negligible contribution:
\begin{lemma}
Assume that $\kappa \geq 1$ and $u$ satisfies \eqref{k0} uniformly in $[0,T]$ with $T\leq\kappa^{-2}$. Then we have
\begin{equation}
\bigg| \int_{-T}^T \int_{-\infty}^\infty \psi^2_{x_0} (j-j_2) \dx \dt \bigg|
\lesssim \kappa^{-1} (T\kappa^2)^\frac14  
\big( \norm{u}_{L^\infty_t{H^{-1}_\kappa}}^2 + \norm{u}_{\ls_\kappa}^2 \!\big)
\label{lem 4p3}
\end{equation}
uniformly in $x_0 \in \R$.
\end{lemma}
\begin{proof}
In order to exhibit cancellation in $j-j_2$, we insert the series \eqref{g} for $g$ into the expression~\eqref{j} for $j$:
\begin{align}
&j-j_2
\nonumber\\
&={} \tfrac{1}{g}( 4\kappa^3 h_1 + u ) +2\kappa R_0(2\kappa) u'' + h_1(16\kappa^5 h_1 + 4\kappa^2 u) 
\label{j3 1} \\
&\qquad + (\tfrac{1}{g} - 2\kappa) 4\kappa^3 h_2 
\label{j3 2} \\
&\qquad + 4\kappa^3 \tfrac{1}{g} \sum_{\ell\geq 3} h_\ell .
\label{j3 3}
\end{align}
We will estimate the contributions \eqref{j3 1}--\eqref{j3 3} one at a time.

Let us start with \eqref{j3 1}.  Using the identity \eqref{h1}, we write
\begin{align*}
\eqref{j3 1}
&= \kappa \big( \tfrac{1}{g} - 2\kappa + 4\kappa^2 h_1 \big) h''_1 \\
&= \kappa \big( \tfrac{1}{g} - 2\kappa + \tfrac{4\kappa^2 h_1}{1 + 2\kappa h_1} \big) h''_1 + \tfrac{8\kappa^4}{1+2\kappa h_1} h_1^2 h''_1 .
\end{align*}
For the first term on the RHS, we use \eqref{1/g est 2} and \eqref{h1 3} to bound:
\begin{align*}
&\bigg| \int_{-T}^T \int_{-\infty}^\infty \psi^2 \kappa \big( \tfrac{1}{g} - 2\kappa + \tfrac{4\kappa^2 h_1}{1 + 2\kappa h_1} \big) h''_1 \dx\dt \bigg| \\
&\qquad\lesssim \kappa \bnorm{ \psi \big( \tfrac{1}{g} - 2\kappa + \tfrac{4\kappa^2 h_1}{1 + 2\kappa h_1} \big) }_{L^2_{t,x}} \norm{ \psi h''_1 }_{L^2_{t,x}} \\
&\qquad \lesssim \kappa^{-\frac32} (T \kappa^2)^{\frac14} \norm{u}_{L^\infty_t{H^{-1}_\kappa}}\big( \norm{u}_{L^\infty_t{H^{-1}_\kappa}}^2 + \norm{u}_{\ls_\kappa}^2 \big) .
\end{align*}
For the second term, we use \eqref{h1 4}, \eqref{h1 1}, and \eqref{h1 3} to bound:
\begin{align*}
\bigg| \int_{-T}^T \int_{-\infty}^\infty \psi^2 \tfrac{8\kappa^4}{1+2\kappa h_1} h_1^2 h''_1 \dx\dt \bigg| 
&
\lesssim \kappa^{\frac72} \bnorm{ \tfrac{1}{1+2\kappa h_1} }_{L^\infty_{t,x}} \norm{h_1}_{L^\infty_tH^1_\kappa} \norm{\psi h_1}_{L^2_{t,x}} \norm{\psi h''_1}_{L^2_{t,x}}
 \\
&
\lesssim \kappa^{-\frac12} T^{\frac12} \norm{u}_{L^\infty_tH^{-1}_\kappa}^2 \big( \norm{u}_{L^\infty_t H^{-1}_\kappa} + \norm{u}_{\ls_\kappa} \big)
 \\
&
\lesssim \kappa^{-\frac32} (T\kappa^2)^{\frac12} \norm{u}_{L^\infty_tH^{-1}_\kappa} \big( \norm{u}_{L^\infty_t H^{-1}_\kappa}^2 + \norm{u}_{\ls_\kappa}^2 \big) .
\end{align*}

Next, we turn to \eqref{j3 2}.  We write
\begin{equation*}
\eqref{j3 2}
= 4\kappa^3 \big( \tfrac{1}{g} - 2\kappa + \tfrac{4\kappa^2 h_1}{1+2\kappa h_1} \big) h_2 - \tfrac{16\kappa^5}{1+2\kappa h_1} h_1h_2 .
\end{equation*}
For the first term on the RHS, we use \eqref{1/g est 2} and \eqref{h2} to estimate
\begin{align*}
&\bigg| \int_{-T}^T \int_{-\infty}^\infty \psi^2 4\kappa^3 \big( \tfrac{1}{g} - 2\kappa + \tfrac{4\kappa^2 h_1}{1+2\kappa h_1} \big) h_2 \dx\dt \bigg| \\
&\qquad\lesssim \kappa^3 \bnorm{ \psi \big(\tfrac{1}{g} - 2\kappa + \tfrac{4\kappa^2 h_1}{1+2\kappa h_1} \big) }_{L^2_{t,x}} \norm{ \psi h_2 }_{L^2_{t,x}} \\
&\qquad\lesssim 
\kappa^{-2} (T \kappa^2)^{\frac12}\norm{u}_{L^\infty_t{H^{-1}_\kappa}}^2 \big( \norm{u}_{L^\infty_t{H^{-1}_\kappa}}^2 + \norm{u}_{\ls_\kappa}^2 \big) \\
&\qquad
{}\lesssim 
\kappa^{-\frac32} (T \kappa^2)^{\frac12}\norm{u}_{L^\infty_t{H^{-1}_\kappa}} \big( \norm{u}_{L^\infty_t{H^{-1}_\kappa}}^2 + \norm{u}_{\ls_\kappa}^2 \big) .
\end{align*}
For the second term, we use \eqref{h1h2} to bound
\begin{align*}
\bigg| \int_{-T}^T \int_{-\infty}^\infty \psi^2 \tfrac{16\kappa^5}{1+2\kappa h_1} h_1h_2 \dx\dt \bigg| 
&\lesssim \kappa^5 \bnorm{ \tfrac{1}{1+2\kappa h_1} }_{L^{\infty}_{t,x}} \norm{ \psi^2 h_1h_2 }_{L^1_{t,x}} \\
&\lesssim \kappa^{-\frac32} (T \kappa^{2})^{\frac34} \norm{u}_{L^\infty_t{H^{-1}_\kappa}} \big( \norm{u}_{L^\infty_t{H^{-1}_\kappa}}^2 + \norm{u}_{\ls_\kappa}^2 \big) .
\end{align*}

Finally, we turn to \eqref{j3 3}.  Using \eqref{h3}, we estimate
\begin{align*}
\bigg| \int_{-T}^T \int_{-\infty}^\infty \psi^2 \eqref{j3 3} \dx \dt \bigg|
&\lesssim \kappa^3 \bnorm{\tfrac{1}{g}}_{L^\infty_{t,x}}\sum_{\ell
\geq 3} \norm{\psi^2 h_\ell}_{L^1_{t,x}} \\
&\lesssim \kappa^{-\frac32}  (T\kappa^2)^{\frac34} \norm{u}_{L^\infty_t{H^{-1}_\kappa}} \big( \norm{u}_{L^\infty_t{H^{-1}_\kappa}}^2 + \norm{u}_{\ls_\kappa}^2 \big) .
\end{align*}

Altogether, we obtain
\begin{equation*}
\bigg| \int_{-T}^T \int_{-\infty}^\infty \psi^2_{x_0} (j-j_2) \dx \dt \bigg|
\lesssim \kappa^{-\frac32} (T\kappa^2)^\frac14 \norm{u}_{L^\infty_t{H^{-1}_\kappa}}  
\big( \norm{u}_{L^\infty_t{H^{-1}_\kappa}}^2 + \norm{u}_{\ls_\kappa}^2 \!\big) .
\end{equation*}
This implies \eqref{lem 4p3} by \eqref{k0}, and thus finishes the proof of the lemma.
\end{proof}

The last step in the proof of the theorem is to prove the estimate \eqref{rhs2}, which shows that the contributions of the source terms in the density flux relation \eqref{alpha micro 2} are perturbative:
\begin{lemma}
Assume that $\kappa \geq 1$ and $u$ satisfies \eqref{k0} uniformly in $[0,T]$, and that  the coefficients $a_1,a_2,a_3,a_4$ in \eqref{gKdV} satisfy \eqref{a}--\eqref{d}. Then
\begin{equation}
\bigg| \int_{-T}^T \int \phi_{x_0}\,  d\rho|_u \big[ (a_1u')' + a_2u^2 + a_3u' + a_4u \big] \dx\dt \bigg|  \lesssim  \kappa^{-1} (\eps +\kappa^{-2} ) \big( \norm{u}_{L^\infty_tH^{-1}_\kappa}^2 + \norm{u}_{\ls_\kappa}^2 \big) 
\label{lem 4p4}
\end{equation}
uniformly in $x_0 \in \R$.
\end{lemma}
\begin{proof}
We begin with the contribution of $(a_1u')'$.  Write
\begin{equation}
a_1(t,x) = c \int a_1(t,x) \psi_z^3(x) \dz 
\label{prtn}
\end{equation}
for some constant $c>0$.  By \eqref{a} we have
\begin{equation*}
\int \norm{ \psi_z^2 a_1u' }_{L^2_tH^{-1}_\kappa}\dz
\lesssim \int \norm{ \psi_z a_1 }_{L^\infty_t X} \norm{ \psi_z u' }_{L^2_tH^{-1}_\kappa} \dz
\lesssim \eps \norm{ u }_{\ls_\kappa} ,
\end{equation*}
where $X = H^1$ or $X = W^{1,\infty}$.  Therefore, \eqref{dg LSk} yields
\begin{align*}
&\bigg| \int_{-T}^T \int \phi\,  d\rho|_u (a_1u')' \dx\dt \bigg| \lesssim \eps \kappa^{-1} 
\big( \norm{u}_{L^\infty_t{H^{-1}_\kappa}} + \norm{u}_{\ls_\kappa} \big) \norm{u}_{\ls_\kappa} ,
\end{align*}
and this is an acceptable contribution to \eqref{lem 4p4}.

For the contributions of $a_2u^2$, $a_3u'$, and $a_4u$, we use \eqref{dg H-1} to estimate
\begin{equation}
\begin{aligned}
&\bigg| \int_{-T}^T \int \phi \,  d\rho|_u  ( a_2u^2 + a_3u' + a_4u ) \dx\dt \bigg| \\
&\qquad\lesssim \kappa^{-1} \norm{u}_{L^\infty_t{H^{-1}_\kappa}} \big( \snorm{a_2u^2}_{L^1_tH^{-1}_\kappa} + \snorm{a_3u'}_{L^1_tH^{-1}_\kappa} + \snorm{a_4u}_{L^1_tH^{-1}_\kappa} \big) .
\end{aligned}
\label{lem 4p4 2}
\end{equation}

For $a_4u$, we use \eqref{d} and \eqref{k0} to bound
\begin{equation*}
\norm{a_4u}_{L^1_tH^{-1}_\kappa} 
\lesssim T \norm{a_4}_{L^\infty_tX} \norm{u}_{L^\infty_t{H^{-1}_\kappa}}
\lesssim  \kappa^{-2} \norm{u}_{L^\infty_t{H^{-1}_\kappa}}  ,
\end{equation*}
where $X = H^1$ or $X = W^{1,\infty}$.

For $a_3u'$, we use \eqref{prtn} and \eqref{c} to bound
\begin{equation*}
\snorm{ a_3u' }_{L^1_tH^{-1}_\kappa}
\lesssim \int \snorm{ \psi^2_z a_3 u' }_{L^1_tH^{-1}_\kappa} \dz
\lesssim \norm{u}_{\ls_\kappa} \int \bnorm{ \psi_z a_3 }_{L^2_tX} \dz 
\lesssim \eps \norm{u}_{\ls_\kappa} ,
\end{equation*}
where $X = H^1$ or $X = W^{1,\infty}$.

Finally, for $a_2u^2$, we use \eqref{LSk 4 new}--\eqref{LSk 2 new} to bound
\begin{equation}
\norm{ \psi_z u}_{L^2_{t,x}}
\lesssim \norm{u}_{\ls_\kappa} + \kappa \norm{u}_{
{L^2_t}H^{-1}_\kappa}
\lesssim \norm{u}_{\ls_\kappa} + \norm{u}_{L^\infty_tH^{-1}_\kappa} ,
\label{LSk 5}
\end{equation}
provided that $T \leq \kappa^{-2}$.  
Using \eqref{prtn}, the embedding $L^1\hookrightarrow H^{-1}_\kappa$, and \eqref{b}, we see that
\begin{align*}
\norm{ a_2 u^2 }_{L^1_tH^{-1}_\kappa}
&\lesssim \int \norm{\psi_z^3 a_2 u^2}_{L^1_tH^{-1}_\kappa} \dz
\lesssim \kappa^{-\frac{1}{2}} \int  \norm{\psi_z^3 a_2 u^2}_{L^1_{t,x}} \!\dz \\
&\lesssim \kappa^{-\frac{1}{2}} \int \norm{ \psi_z u}_{L^2_{t,x}}^2 \norm{\psi_z a_2}_{L^\infty_{t,x}} \dz 
\lesssim \eps \kappa^{-\frac{1}{2}} \big( \norm{u}_{L^\infty_tH^{-1}_\kappa}^2 + \norm{u}_{\ls_\kappa}^2 \big) .
\end{align*}

Returning to \eqref{lem 4p4 2}, we conclude
\begin{align*}
&\bigg| \int_{-T}^T \int \phi \,  d\rho|_u  ( a_2u^2 + a_3u' + a_4u ) \dx\dt \bigg| \\
&\qquad\lesssim \kappa^{-1} ( \eps + \kappa^{-2} ) \big( \norm{u}_{L^\infty_tH^{-1}_\kappa}^2 + \norm{u}_{\ls_\kappa}^2 \big) 
+ \eps \kappa^{-\frac{3}{2}} \norm{u}_{L^\infty_tH^{-1}_\kappa} \big( \norm{u}_{\ls_\kappa}^2 + \norm{u}_{L^\infty_tH^{-1}_\kappa}^2 \big) .
\end{align*}
This implies \eqref{lem 4p4} by \eqref{k0}, and thus completes the proof.

\end{proof}
Collecting the last three lemmas, this concludes the proof of Theorem~\ref{t:local smoothing}.

\end{proof}

\section{The a-priori estimates}
\label{s:ap est}

In this section, we will prove the energy estimate for solutions to \eqref{gKdV}, and use it to conclude the proof of our main
results in Theorem~\ref{t:energy}
and Theorem~\ref{t:le}.

\begin{proposition}
\label{t:ap est}
For any $\eps\in (0,1]$, if the coefficients $a_1,a_2,a_3,a_4$ of \eqref{gKdV} satisfy \eqref{a}--\eqref{d}, then any solution $u(t)$ to \eqref{gKdV} satisfies
\begin{align}
\norm{u}_{C_tH^{-1}_\kappa}^2
&\lesssim \norm{u(0)}_{H^{-1}_\kappa}^2 +  (\eps+T) \big( \norm{u}_{L^\infty_tH^{-1}_\kappa}^2 + \norm{u}_{\ls_\kappa}^2 \big) 
\label{alpha dot 2}
\end{align}
uniformly for $T \leq \kappa^{-2}$ and  $\kappa$ satisfying \eqref{k0} in $[0,T]$.
\end{proposition}

\begin{proof}
Integrating \eqref{alpha micro 2} in space and using \eqref{drho}, we obtain
\begin{equation*}
\ddt \alpha = \int \big[ \tfrac{1}{2g^2} dg|_u + 2\kappa R_0(2\kappa) \big] \big[ (a_1u')' + a_2u^2 + a_3u' + a_4u \big] \dy .
\end{equation*}
Next, we use \eqref{dg} and the identity
\begin{equation*}
\int \frac{ G(x,y) G(y,x) }{2 g(y)^2 } \dy = g(x)
\end{equation*}
(see~\cite{Killip2019}*{Lem.\ 2.5} for a proof) to write
\begin{equation*}
\ddt \alpha = - \int (g-\tfrac{1}{2\kappa}) \big[ (a_1u')' + a_2u^2 + a_3u' + a_4u \big] \dx .
\end{equation*}

By \eqref{alpha est} and the fundamental theorem of calculus, this yields
\begin{equation}
\norm{u}_{C_tH^{-1}_\kappa}^2
\lesssim \norm{u(0)}_{H^{-1}_\kappa}^2 + \kappa \int_{-T}^T \bigg| \int (g-\tfrac{1}{2\kappa}) \big[ (a_1u')' + a_2u^2 + a_3u' + a_4u \big] \dx \bigg|\dt .
\label{alpha dot 3}
\end{equation}
We will estimate the contribution from each coefficient $a_j$ in turn.

We begin with the contribution of $a_4$.  Using \eqref{g est} and \eqref{d}, we have
\begin{align*}
\int_{-T}^T \bigg| \int (g-\tfrac{1}{2\kappa}) a_4u \dx \bigg|\dt
&\leq T \norm{ g - \tfrac{1}{2\kappa} }_{L^\infty_tH^1_\kappa} \norm{ a_4u }_{L^\infty_tH^{-1}_\kappa} \\
&\lesssim \kappa^{-1} T \norm{a_4}_{L^\infty_tX}  \norm{u}_{L^\infty_tH^{-1}_\kappa}^2 \\
&\lesssim \kappa^{-1} T \norm{u}_{L^\infty_tH^{-1}_\kappa}^2 ,
\end{align*}
where $X = H^1$ or $X = W^{1,\infty}$.

Next, we address $a_3$.  Using the continuous partition of unity \eqref{prtn} and \eqref{c}, we have
\begin{align*}
\int_{-T}^T \bigg| \int (g-\tfrac{1}{2\kappa}) a_3u' \dx \bigg| \dt
&\leq \norm{ g - \tfrac{1}{2\kappa} }_{L^\infty_tH^1_\kappa} \norm{ a_3 u' }_{L^1_tH^{-1}_\kappa} \\
&\lesssim \kappa^{-1} \norm{u}_{L^\infty_t H^{-1}_\kappa} \int \norm{ \psi_z a_3 }_{L^2_tX} \norm{\psi_z u'}_{L^2_tH^{-1}_\kappa} \dz  \\
&\lesssim \eps \kappa^{-1} \norm{u}_{L^\infty_t H^{-1}_\kappa} \norm{u}_{\ls_\kappa} .
\end{align*}

For $a_2$, we use the embedding $H^1_\kappa \hookrightarrow L^\infty$, \eqref{prtn}, and \eqref{LSk 5} to bound
\begin{align*}
\int_{-T}^T \bigg| \int (g-\tfrac{1}{2\kappa}) a_2u^2 \dx\bigg| \dt
&\leq \norm{ g - \tfrac{1}{2\kappa} }_{L^\infty_{t,x}} \norm{ bu^2 }_{L^1_{t,x}} \\
&\lesssim \kappa^{-\frac{3}{2}} \norm{ u }_{L^\infty_tH^{-1}_\kappa} \int \norm{\psi_z a_2}_{L^\infty_{t,x}} \norm{\psi_z u}_{L^2_{t,x}}^2 \dz \\
&\lesssim \eps \kappa^{-\frac{3}{2}} \norm{ u }_{L^\infty_tH^{-1}_\kappa} \big( \norm{u}_{\ls_\kappa}^2 +  \norm{u}_{L^\infty_tH^{-1}_\kappa}^2  \big) 
 \\
&\lesssim \eps \kappa^{-1} \big( \norm{u}_{\ls_\kappa}^2 +  \norm{u}_{L^\infty_tH^{-1}_\kappa}^2  \big) ,
\end{align*}
using \eqref{k0} at the last step.

Lastly, we turn to the contribution of $a_1$.  We integrate by parts once in space, and then we use \eqref{prtn}, \eqref{LSk g}, and \eqref{a} to estimate:
\begin{align*}
\int_{-T}^T \bigg|  \int (g-\tfrac{1}{2\kappa}) (a_1u')' \dx \bigg|  \dt
&\leq \int \norm{ \psi_z g' }_{L^2_tH^1_\kappa} \norm{ \psi_z u' }_{L^2_tH^{-1}_\kappa} \norm{ \psi_z a_1 }_{L^\infty_t X} \dz  \\
&\lesssim \eps \kappa^{-1}  \norm{u}_{\ls_\kappa}^2 .
\end{align*}

Altogether, returning to \eqref{alpha dot 3}, we obtain
\begin{align*}
\norm{u}_{C_tH^{-1}_\kappa}^2 
&\lesssim \norm{u(0)}_{H^{-1}_\kappa}^2 +(\eps +T) \big( \norm{u}_{L^\infty_tH^{-1}_\kappa}^2 + \norm{u}_{\ls_\kappa}^2 \big)  . 
\qedhere
\end{align*}
\end{proof}

At this point, we are prepared to complete the proofs of Theorems~\ref{t:energy} and \ref{t:le} by combining Propositions~\ref{a} and \ref{alpha dot 2} with a continuity argument.
We restate the results together in a stronger
form:

\begin{theorem}
\label{t:ap est full}
There exist $\epsilon_0, c_0 > 0$ so that, given coefficients $a_1,\dots,a_4$ in \eqref{gKdV} satisfying \eqref{a}--\eqref{d} with $\eps = \eps_0$ and $R \geq 1$, for any initial data $u_0$ satisfying
\begin{equation}
 \|u_0\|_{H^{-1}} \leq R ,
\end{equation}
the solution $u(t)$ to \eqref{gKdV} exists up to time
\begin{equation}
T_0 = c_0 R^{-4} ,
\end{equation}
and on $[0,T_0]$ it satisfies
\begin{equation}
\norm{u}_{L^\infty_tH^{-1}_\kappa} + \norm{u}_{\ls_\kappa}  \lesssim R, \qquad \kappa \gg R^{2} .
\label{ap est 2}
\end{equation}
\end{theorem}

\begin{proof}
For any $\kappa \gg R^{2}$, the condition 
\eqref{k0} holds at $t = 0$ and thus on some time interval $[0,T]$. Let $T  \leq  \kappa^{-2}$ be any time for which the solution exists in $[0,T]$ and 
satisfies \eqref{k0} uniformly in $[0,T]$.
Then both Propositions~\ref{a} and \ref{alpha dot 2} apply in $[0,T]$.

Consider the quantity
\begin{equation*}
B_T := \norm{u}_{C_tH^{-1}_\kappa([-T,T]\times\R)}^2 + \tfrac{1}{4}\norm{u}_{\ls_\kappa([-T,T]\times\R)}^2 .
\end{equation*}
Adding \eqref{LSk} and \eqref{alpha dot 2}, we see that there exists a universal constant $C\geq 1$ so that
\begin{equation}
B_T \leq C R^2 + C(\eps_0 + (T \kappa^2)^\frac14 + \kappa^{-2})B_T .
\label{boot 2}
\end{equation}
Assuming 
\begin{equation}\label{need}
\epsilon_0 \ll 1, \qquad T \ll \kappa^{-2}, 
\qquad \kappa \gg 1,
\end{equation}
this will allow us to conclude
\begin{equation}\label{good bound}
B_T \leq 2C R^2 .
\end{equation}

Next, we will fix our parameters in order.  First we choose $\epsilon_0 \ll 1$ so that the first relation \eqref{need} holds, then
\[
\kappa =1 + 4CR^2
\]
so that the third relation in \eqref{need} holds, and lastly
\[
T_0 = c \kappa^{-2}, \qquad c \ll 1
\]
so that the second relation of \eqref{need} holds. 

Finally, we run a standard continuity argument. We let $T \in (0,T_0]$ be the maximal time so that \eqref{k0} holds in $[0,T]$. Then, the argument above demonstrates that \eqref{good bound} holds in $[0,T]$. In particular, for this choice of $\kappa$, we know that \eqref{k0} holds strictly at time $T$. However, this contradicts the maximality of $T$, unless
$T=T_0$. We conclude that $T = T_0$, and thus \eqref{good bound} holds in $[0,T_0]$. This concludes the proof of the theorem.
\end{proof}

Next, we will show that the assumptions \eqref{hyp 1}--\eqref{hyp 2} on the coefficients $a_1,\dots,a_4$ in Theorems~\ref{t:energy} and \ref{t:le} provide an example of when our hypotheses \eqref{a}--\eqref{d} are satisfied.  
\begin{lemma}
Given $\eps,T\in (0,1]$, there exists a constant $\del\in (0,1]$ so that, if the coefficients $a_1,\dots,a_4$ satisfy
\begin{align}
|a_j(t,x)| + |\partial_xa_j(t,x)| &\leq \del (1+x^2)^{-1} \quad\text{for }j=1,2,3 , 
\label{hyp 3}\\
|a_4(t,x)| + |\partial_xa_4(t,x)| &\lesssim (1+x^2)^{-1}
\label{hyp 4}
\end{align}
uniformly for $|t|\leq T$ and $x\in\R$, then \eqref{a}--\eqref{d} are satisfied for $\eps$.
\end{lemma}

\begin{proof}
We partition $\psi$ as follows:
\begin{equation*}
\psi^2(x) = \operatorname{sech}^2 x \leq \sum_{n\geq 0} \operatorname{sech}^2(n)\, 1_{\{n\leq |x| < n+1\}}(x) .
\end{equation*}
This yields
\begin{equation*}
\int |\psi_z(x)a_j(x)|^2 \dx
\leq \sum_{n\geq 0} \operatorname{sech}^2(n) \int_{n \leq |x-z| < n+1} |a_1(x)|^2\dx
\lesssim \del^2 \langle z \rangle^{-4} ,
\end{equation*}
and so
\begin{equation*}
\int \norm{ \psi_z a_j }_{L^2_x} \dz
\lesssim \del\int \langle z \rangle^{-2} \dz
\lesssim \del .
\end{equation*}
Using the same estimates for $\partial_xa_j$, we obtain
\begin{equation*}
\int \norm{ \psi_z a_j }_{H^1_x} \dz \lesssim \del
\quad\tx{for }j=1,2,3,4.
\end{equation*}
For $j=4$, this shows that \eqref{d} is satisfied.  For $j=1,2,3$, this demonstrates that we may choose $\del\leq 1$ sufficiently small so that \eqref{a}--\eqref{c} hold.
\end{proof}

Lastly, we show that our result applies to the model \eqref{KdVvb} for the propagation of waves in a channel over a variable bottom.
\begin{proof}[Proof of Corollary \ref{t:KdV vb}]
If the function $c:\R\to\R$ that describes the channel bottom is smooth and satisfies $\norm{c}_{L^\infty}<1$, then the function $y:\R\to\R$ given by
\begin{equation*}
y(x) = \int_0^x \frac{1}{b^{\frac53}(\xi)}\,d\xi = \int_0^x \frac{1}{[1-c(\xi)]^{\frac56}}\,d\xi
\end{equation*}
is well-defined and satisfies
\begin{equation}
|y'(x)| \approx 1 \quad\text{uniformly for }x\in\R .
\label{y 1}
\end{equation}
Consequently,
\begin{equation}
\int |f(y)|^2\dy \approx \int |(f\circ y)(x)|^2\dx \quad\text{and}\quad \int |f'(y)|^2\dy \approx \int |(f\circ y)'(x)|^2\dx ,
\label{y 2}
\end{equation}
and so
\begin{equation}
\norm{f\circ y}_{H^1_\kappa} \approx \norm{f}_{H^1_\kappa} \quad\text{uniformly for }\kappa\geq 1.
\label{y 3}
\end{equation}
Additionally, we note that $y^{-1}:\R\to\R$ exists by the inverse function theorem and also satisfies \eqref{y 1}--\eqref{y 3}.  A straightforward duality argument then shows
\begin{equation}
\norm{g\circ y}_{H^{-1}_\kappa} \approx \norm{g}_{H^{-1}_\kappa} \quad\text{uniformly for }\kappa\geq 1.
\label{y 4}
\end{equation}

Making the change of variables
\begin{equation}
u(t,x) = b^{\frac53}(x) v(t,y(x)-4t) ,
\label{cov}
\end{equation}
we find that if $u(t,x)$ solves~\eqref{KdVvb} then $v(t,y)$ solves \eqref{gKdV} with coefficients
\begin{align*}
a_1(t,y) &= 0,\\
a_2(t,y) &= 10b^{\frac23}b' \circ y^{-1} ,\\
a_3(t,y) &= \big[ \tfrac59 b^{\frac43} (b')^2 -\tfrac{10}{3}b^{\frac73}b'' + 4(1-b^{-\frac23}) \big] \circ y^{-1} , \\
a_4(t,y) &= \big[ \tfrac{10}{3} b^2(b')^3 - \tfrac{10}{3} b^3b'b'' - \tfrac{5}{3} b^4b''' -\tfrac{38}{3} b' \big] \circ y^{-1} .
\end{align*}
Clearly, we may choose $\eta\in (0,1]$ sufficiently small so that the conditions
\begin{equation*}
|\partial_x^jc(x)| \leq \eta (1+x^2)^{-1} \quad\text{for }j=0,1,\dots,4 
\end{equation*}
imply that the coefficients $a_1,\dots,a_4$ above satisfy \eqref{hyp 3}--\eqref{hyp 4}.  Note that the $4$ in \eqref{cov} contributes the constant term $4$ in $a_3$, which is needed to ensure that this coefficient vanishes as $y\to\pm\infty$.

This demonstrates that $v$ satisfies the a-priori estimate \eqref{ap est 2}.  This in turn implies that $u$ satisfies \eqref{ap est 2} as well, by \eqref{y 4}.
\end{proof}

\bibliographystyle{amsplain}
\bibliography{refs}

\end{document}